\documentclass[11pt,a4paper]{article}
\usepackage[dvips]{graphicx}
\usepackage{amssymb,latexsym,amsmath}
\usepackage{psfrag}
\usepackage[active]{srcltx} %comando para voltar do dvi para o tex
\input epsf

\usepackage[ansinew]{inputenc} %para usar caracteres latinos, acentos, etc.

\usepackage{color}  %para usar los comandos de color

\newtheorem{prop}{Proposition}[section]
\newtheorem{theo}[prop]{Theorem}

\newtheorem{defn}[prop]{Definition}
\newtheorem{rem}[prop]{Remark}
\newtheorem{rems}[prop]{Remarks}

\newenvironment{proof}
 {\begin{trivlist} \item[\hskip \labelsep {\bf Proof}\hspace*{3 mm}]}
 {\hfill$\Box$\end{trivlist}}

\newcommand{\R}{\mathbb{R}}
\newcommand{\auz}{a_{10}}
\newcommand{\auu}{a_{11}}
\newcommand{\adz}{a_{20}}
\newcommand{\adu}{a_{21}}
\newcommand{\add}{a_{22}}
\newcommand{\att}{a_{33}}
\newcommand{\atu}{a_{31}}
\newcommand{\atd}{a_{32}}
\newcommand{\atz}{a_{30}}
\newcommand{\q}[3][2]{#2_{#3}}
\newcommand{\der}[3][2]{\frac{\partial #2}{\partial #3}}

\begin{document}

\title{Geometric deformations of implicit curves}
\author{Marco Antônio do Couto Fernandes \footnote {Work partially supported by the FAPESP Thematic Project 2019/07316-0.} \, and Samuel P. dos Santos \footnote {Work supported by the postdoctoral grant from FAPESP, No. 2022/10370-9, under the direction of Farid Tari.}}

\maketitle
\begin{abstract}
Let \( f = 0 \) be an implicit singular plane curve. When deforming $f=0$, inflections and vertex emerge from the singularities. In this papper, we classify the deformations of $f=0$ with respect to the inflections and the vertices in the cases of codimension less than or equal to 2, that is, in the cases that occur generically in families of implicit curves with 2 parameters.

\end{abstract}

\renewcommand{\thefootnote}{\fnsymbol{footnote}}
 \footnote[0]{Keywords and Phrase:
Singular curve, inflexion, vertex, geometric deformation.}

%%%%%%%%%%%%%%%%%%%%%%%%%%%%%%%%%%%%%%%%%%%%%%%%%%%%%%%%%%%%%%%%
%%%%%%%%%%%%%%%%%%%%%%%%%%%%%%%%%%%%%%%%%%%%%%%%%%%%%%%%%%%%%%%%
%%%%%%%%%%%%%%%%%%%%%%%%%%%%%%%%%%%%%%%%%%%%%%%%%%%%%%%%%%%%%%%%
\section{Introduction}

In this paper, we consider plane curves implicitly defined by $f=0$, where $f:(\mathbb{R}^2,0)\to (\mathbb{R},0)$ is a singular germ, and the geometry of a versal deformation $F:(\mathbb{R}^2\times\mathbb{R}^p,0)\to (\mathbb{R},0)$ with $p$ parameters. We concentrate the study on vertices and inflexions (for a study in the same sense, but for parametrized plane curves, see \cite{FRS}), that is, we will study how many of these points and in which positions they appear when we deform $f=0$ using $F$.

Singular plane curves can be related to regular curves. For instance, an orthogonal projection of a regular space curve to a plane is singular if the direction of projection is tangent to the space curve. Geometric information on the projected curve provides geometric information on the space
curve itself \cite{David,OT,DN,Wall}.

This papper will be a generalization of the work by Giblin and Diatta \cite{GD}, where they consider the deformations of $f=0$ obtained by the sections of the graph of $f$ by planes parallel to the tangent plane at the origin. Therefore, a directly application of this problem is plane sections of surfaces, which play an important role in, for example, Computer Vision, Shape Analysis, etc.

The paper is organized as follows. In Section 2, we introduce preliminary concepts such as the differential geometry of curves, the k-jets space and the Monge-Taylor map. In Section 3, we define the FRS-equivalence. The curves with singularities $A_1^-$, $A_1^+$, and $A_2$ will be discussed in Sections 4, 5, and 6, respectively.

\section{Preliminaries} \label{sec:prel}

Let $\gamma: I \to \R^2$ be an analytic and regular curve. If $\gamma(t)=(x(t),y(t))$, 
then the curvature of the curve is given by 
\begin{equation}\label{eq:k}
	\kappa(t) = \frac{x' y''-y' x''}{((x')^2+(y')^2)^{3/2}}(t),
\end{equation}
where $'$ means the derivative with respect to $t$.

A point $t_0$ is a \textit{inflection} of the curve $\gamma$ if $\kappa(t_0) = 0$ and it is a \textit{vertex} of $\gamma$ if $\kappa'(t_0) = 0$. 
Inflections and vertices can be captured by considering the contact of $\gamma$ with lines and circles, respectively (see \cite{BG}). 

We observe that a regular point $t_0 \in I$ is an inflection (resp. vertex) of $\gamma$ if and only if $i_\gamma (t_0) = 0$ (resp. $v_\gamma (t_0) = 0$ ), where
$$\begin{array}{ccl}
	i_\gamma & = & x' y''-y' x'',\vspace{0.2cm}\\
	v_\gamma & = & ((x')^2+(y')^2) (x' y'''-y' x''')+3(x' x''+y' y'') (x'' y'-y'' x').
\end{array}$$

We say that $t_0 \in I$ is an inflection (resp. vertex) of \textit{order} $n$ when $\kappa(t_0) = 0$, $\kappa^{(i)}(t_0) = 0$ for $1 \leq i \leq n-1$ and $\kappa^{(n)}(t_0) \neq 0$ (resp. $\kappa^{(i)}(t_0) = 0$ for $1 \leq i \leq n$ and $\kappa^{(n+1)}(t_0) \neq 0$), where $\kappa^{(i)}$ represents the $i$-th derivative of $\kappa$. An inflection or vertex of order 1 is called \textit{ordinary} or \textit{simple}.

A non-constant  germ of analytic function $f : (\R^2,0) \to (\R,0)$ defines a germ of analytic curve $C_f$ in $(\R^2,0)$ by
$$
C_f = \{(x,y) \in (\R^2,0) : f(x,y)=0 \}.
$$ 
Consider the decomposition of $f$ into irreducible factors in $\mathcal{E}_2$,
$f = f_1 \cdot f_2 \cdots f_n.$
Then $C_f = C_{f_1} \cup C_{f_2} \cup \cdots \cup C_{f_n}$. The curves $C_{f_i}$ are called \textit{branches} of $C_f$. Furthermore, if we write
$$f= F_m+F_{m+1}+\cdots,$$
where each $F_i$ is a homogeneous polynomial of degree $i$ in the variables $x$ and $y$ and $F_m \neq 0$, then $m$ is the \textit{multiplicity} of $f$ and $C_{ F_m}$ is the \textit{tangent cone} of $C_f$. We say that the origin of $\R^2$ is an \textit{umbilic point} of $f=0$ when the origin of $\R^3$ is an umbilic point on the surface $z = f(x,y)$.

Given a non-constant analytic function $f$,  we can calculate 
the curvature $\kappa$ and its derivatives at regular points of $C_f$ 
using the implicit function theorem.
Taking the numerators of these function, 
we say that
$p_0 = (x_0,y_0)$ is an inflection of $C_f$ if $p_0$ is a solution of the system
\begin{equation}\label{eq:sis_I}
	\left\{\begin{array}{l}
		f(x,y) = 0,\\
		i_f(x,y) = 0,
	\end{array}\right.
\end{equation}
with $i_f = f_{xx} f_y^2-2 f_{xy} f_x f_y+f_{yy} f_x^2$, and an inflection of order 2 if $p_0$ is a solutions of the system
\begin{equation}\label{eq:sis_I2}
	\left\{\begin{array}{l}
		f(x,y) = 0,\\
		i_{f}(x,y) = 0,\\
		i_{2,f}(x,y) = 0,
	\end{array}\right.
\end{equation}
with $$i_{2,f} = -f_x^3f_{yyy}+3f_x^2f_y f_{xyy}-3f_xf_y^2 f_{xxy}+f_y^3f_{xxx},$$
where subscripts denote partial derivatives.
We say that $p_0$ is a vertex of $C_f$ when $p_0$ is a solution of the system
\begin{equation}\label{eq:sis_V}
	\left\{\begin{array}{l}
		f(x,y) = 0,\\
		v_f(x,y) = 0,
	\end{array}\right.
\end{equation}
where
$$\begin{array}{ccl}
	v_f & = & (f_x^2+f_y^2) \left( -f_{x}^3 f_{yyy}+3 f_{x}^2 f_{y} f_{xyy} -3 f_{x} f_{y}^2 f_{xxy} +f_{y}^3 f_{xxx} \right)\vspace{0.2cm}\\
	& & +3 f_{x} f_{y} (f_{yy}-f_{xx}) \left(f_{x}^2 f_{yy}+f_{xx} f_{y}^2\right)\vspace{0.2cm}\\
	& & +6 f_{x} f_y f_{xy}^2 \left(f_{y}^2-f_{x}^2\right)\vspace{0.2cm}\\
	& & +3 f_{xy} \left(f_{x}^4 f_{yy}+3 f_{x}^2 f_{y}^2 (f_{xx}-f_{yy})-f_{xx} f_{y}^4\right),
\end{array}$$
and a vertex of order 2 if $p_0$ is a solutions of the system
\begin{equation}\label{eq:sis_I22}
	\left\{\begin{array}{l}
		f(x,y) = 0,\\
		v_{f}(x,y) = 0,\\
		v_{2,f}(x,y) = 0,
	\end{array}\right.
\end{equation}
with $$\begin{array}{ccl}
	v_{2,f} &= &2 (f_x^2 + f_y^2) (-6 f_x^4 f_{xy} f_{xyyy} f_y + 12 f_x^3 f_{xxy} f_{xy} f_y^2 - 
	3 f_x^2 f_{xx} f_{xxy} f_y^3 - 6 f_x^2 f_{xxxx} f_{xy} f_y^3 \\ & &+
	2 f_x f_{xx} f_{xxx} f_y^4 + 6 f_x f_{xxy} f_{xy} f_y^4 - 6 f_x f_{xx} f_{xyy} f_y^4 + 
	3 f_{xx} f_{xxy} f_y^5 - 4 f_{xxx} f_{xy} f_y^5\\& & + 3 f_x^5 f_{xyy} f_{yy} - 
	6 f_x^4 f_{xxy} f_y f_{yy} + 3 f_x^3 f_{xxx} f_y^2 f_{yy} + 
	3 f_x^3 f_{yyy} f_y^2 f_{yy} - 6 f_x^2 f_{xxy} f_y^3 f_{yy} \\& &+ 
	3 f_x f_{xxx} f_y^4 f_{yy} + f_x^4 f_{xx} f_y f_{yyy} - 2 f_x^3 f_{xy} f_y^2 f_{yyy}+ 
	3 f_x^2 f_{xx} f_y^3 f_{xyyy})\\& & - (f_x^2 + f_y^2)^2 (-4 f_x^3 f_{xyyy} f_y + 
	6 f_x^2 f_{xxyy} f_y^2 - 4 f_x f_{xxy} f_y^3 + f_{yyxx} f_y^4 + f_x^4 f_{yyyy})\\& & - 
	3 (-2 f_x f_{xy} f_y + f_{xx} f_y^2 + f_x^2 f_{yy}) (-f_x^2 f_{xx}^2 f_y^2 - 
	4 f_x f_{xx} f_{xy} f_y^3 - f_{xx}^2 f_y^4 - 8 f_{xy}^2 f_y^4 \\& &+ f_x^4 f_{xx} f_{yy} + 
	6 f_x^2 f_{xx} f_y^2 f_{yy} + 12 f_x f_{xy} f_y^3 f_{yy} + f_{xx} f_y^4 f_{yy} - 
	f_x^4 f_{yy}^2 - 5 f_x^2 f_y^2 f_{yy}^2).
\end{array}$$ 

\begin{rems}\label{rem:inf_vert}
	(1) Clearly, if $p_0$ is a singularity of $C_f$, then $f_x(p_0) = f_y(p_0) = 0$ and $p_0$ satisfies {\rm (\ref{eq:sis_I})} and {\rm (\ref{eq:sis_V})}, so it will be considered as an inflection and a vertex of $C_f$.
	
	(2) We can define inflections and vertices of order $n$ for $C_f$ as in the case of parameterized curves.
		
	(3) We will consider, without loss of generality, the point of interest $p_0$ to be the origin.
\end{rems}

\subsection{Monge-Taylor map}

Given a smooth function germ $f : (\R^2,0) \to \R$, the $k$\textit{-jet} of $f$ in $(x_0,y_0) \in (\R^2,0)$, denoted by $j^kf(x_0,y_0)$, is the Taylor polynomial of degree $k$ of $f$ in $(x_0,y_0)$, that is,
$$j^kf(x_0,y_0) = f(x_0,y_0)+\sum_{i=1}^{k} \sum_{j=0}^{i} \frac{\frac{\partial^i f}{\partial x^{i-j} \partial y^j}(x_0,y_0)}{(i-j)! j!}(x-x_0)^{i-j} (y-y_0)^j.$$
Note that we consider the constant term in the $k$-jet.

The set of all $k$-jets of germs $f:(\R^2,0) \to \R$ is denoted by $J^k(2,1)$, i.e., $J^k( 2,1)$ is the set of polynomials in two variables with degree less than or equal to $k$. We can identify $J^k(2,1)$ with $\R^{(k+2)(k+1)/2}$ so that
$$c_{00}+c_{10} x+c_{11} y+\cdots c_{kk} y^k \in J^k(2,1) \leftrightarrow (c_{00},c_{10},c_{11},\cdots,c_{kk}) \in \R^{(k+2)(k+1)/2}.$$ 

Given a smooth function germ $f: (\R^2,0)\to \R$, the \textit{Monge-Taylor map} is defined by
$$\begin{array}{ccccl}
	\Phi_f & : & (\R^2,0) & \to & J^k(2,1)\\
	&& (x,y) & \mapsto & j^kf(x,y).
\end{array}$$
Therefore, the image of $\Phi_f$ is a submanifold of $J^k(2,1)$ of dimension two.
%Two germs $f_i : (\R^n,0) \to (\R,0)$, with $i=1,2$, are said to be $\mathcal{R}$-\textit{equivalent} when there is a germ of a diffeomorphism $ h:(\R^n,0) \to (\R^n,0)$ such that $f_1 = f_2 \circ h^{-1}$. The simple singularities of  germs of function are classified by Arnold and, when $n = 2$, are $\mathcal{R}$-equivalent to the following normal forms
%$$A_k: \pm(x^2 \pm y^{k+1}), k \geq 1; \, D_k: x^2 y\pm y^{k+1}, k \geq 4; \, E_6: x^3+y^4; \, E_7: x^3+x y^3; \, E_8: x^3+y^5.$$
%The \textit{Milnor number} of a germ $f : (\R^2,0) \to (\R,0)$, denoted by $\mu(f)$, is the codimension of the $\mathcal{R}$-orbit of $f$ in $\mathcal{O}_2$ and is given  by $\mu(f) = m(f_x,f_y)$.

\subsection{Stratification of k-jet space}

Let $f:(\mathbb{R}^2,0) \to (\mathbb{R},0)$ be a smooth function germ with
\begin{equation}\label{eq:fA1-naosimplif}
	j^k f(x,y)=\sum_{i=0}^{k} \sum_{j=0}^i \q{c}{ij}x^{i-j}y^j.
\end{equation} 
Throughout the text, we will be interested in curves implicitly defined by $f(x,y) = 0$. Such curves can be described by $\Phi_f(x,y) \in \mathcal{L}_0$, where
$$\mathcal{L}_0 = \{(c_{00},c_{10}, \cdots, c_{kk}) \in J^k(2,1) : c_{00} = 0\}.$$
Note that $\mathcal{L}_0$ is a regular submanifold of $J^k(2,1)$ with codimension 1.

The origin will be a singular point of $f$ when $c_{10} = c_{11} = 0$. Thus, the submanifold
$$\mathcal{S} = \{(c_{00},c_{10}, \cdots, c_{kk}) \in J^k(2,1) : c_{10} = c_{11} = 0\}$$
of $J^k(2,1)$ determines the singular points, that is, $f$ is singular in $(x_0,y_0)$ if and only if $\Phi_f (x_0,y_0) \in \mathcal{S}$. The singularity is of Morse type when $c_{21}^2-4 c_{20} c_{22} \neq 0$, and we obtain the open and dense set $\mathcal{S}_M$ of $\mathcal{S}$ which is given by
$$\mathcal{S}_M = \{(c_{00},c_{10}, \cdots, c_{kk}) \in \mathcal{S} : c_{21}^2-4 c_{20} c_{22} \neq 0\}.$$
It is easy to see that $\mathcal{S}$ and $\mathcal{S}_M$ are regular submanifolds of $J^k(2,1)$ with codimension 2.

It follows from the system (\ref{eq:sis_I}) that the origin will be an inflection of $C_f$ when
$$i_f(0,0) = 2(c_{20} c_{11}^2-c_{21} c_{10} c_{11}+c_{22} c_{10}^2) = 0.$$

This equality defines a submanifold $\mathcal{I}$ in $J^k(2,1)$, given by
$$\mathcal{I} = \{(c_{00},c_{10}, \cdots, c_{kk}) \in J^k(2,1) : i(c_{10},c_{11},\cdots,c_{22}) = 0\},$$
where
$$i(c_{10},c_{11},\cdots,c_{22}) = c_{20} c_{11}^2-c_{21} c_{10} c_{11}+c_{22} c_{10}^2,$$
called submanifold of inflections. Likewise, using (\ref{eq:sis_V}), we have the submanifold of vertices given by
$$\mathcal{V} = \{(c_{00},c_{10}, \cdots, c_{kk}) \in J^k(2,1) : v(c_{10},c_{11},\cdots,c_{33}) = 0\},$$
with
$$\begin{array}{cl}
	v(c_{10},c_{11},\cdots,c_{33}) = & -2 c_{10} c_{11}^3 c_{20}^2 + 3 c_{10}^2 c_{11}^2 c_{20} c_{21} - c_{11}^4 c_{20} c_{21} -  c_{10}^3 c_{11} c_{21}^2\\
	& + c_{10} c_{11}^3 c_{21}^2 - 2 c_{10}^3 c_{11} c_{20} c_{22} +  2 c_{10} c_{11}^3 c_{20} c_{22} + c_{10}^4 c_{21} c_{22} \\
	& - 3 c_{10}^2 c_{11}^2 c_{21} c_{22} +  2 c_{10}^3 c_{11} c_{22}^2 + c_{10}^2 c_{11}^3 c_{30} + c_{11}^5 c_{30} \\
	& - c_{10}^3 c_{11}^2 c_{31} -  c_{10} c_{11}^4 c_{31} + c_{10}^4 c_{11} c_{32} + c_{10}^2 c_{11}^3 c_{32} - c_{10}^5 c_{33} \\
	& -  c_{10}^3 c_{11}^2 c_{33}.
\end{array}$$
Note that $\mathcal{S} \subset \mathcal{I}$ and $\mathcal{S} \subset \mathcal{V}$, which is consistent with Remark \ref{rem:inf_vert} (1).

\begin{prop}\label{prop:IV}
	Let $f:(\mathbb{R}^2,0) \to (\mathbb{R},0)$ be a smooth function germ and  $P = \Phi_f(0,0)$. It follows that:
	\begin{enumerate}
		\item If $f$ has a $A_1^+$-singularity at the origin and the origin is not a umbilic point of $f$, then, locally at $P$, $\mathcal{I}$ (resp. $\mathcal{V}$) coincides with $\mathcal{S}$ (resp. is formed by two regular submanifolds of codimension one that are transversal in $\mathcal{S}$).
		\item If $f$ has a $A_1^-$-singularity at the origin, then, locally at $P$, $\mathcal{I}$ (resp. $\mathcal{V}$) is formed by two (resp. four) regular submanifolds of codimension 1 that are transversals in $\mathcal{S}$. Furthermore, two of the submanifolds of $\mathcal{V}$ are tangent to the submanifolds of $\mathcal{I}$ at $P$.
	\end{enumerate}
\end{prop}

\begin{proof}
	Suppose that $f$ has a $A_1^+$-singularity at the origin. Fix $c_{20},c_{21},c_{22}$ in a neighborhood of $P$ and consider the function
	$$g(c_{10},c_{11}) = i(c_{10},c_{11},\cdots,c_{22}).$$ Thus $c_{21}^2-4 c_{20} c_{22} < 0$ and
	$$g(c_{10},c_{11}) = 0 \quad \Leftrightarrow \quad c_{10} = c_{11} = 0.$$
	Therefore, $\mathcal{I}$ coincides with $\mathcal{S}$.
	
	However, if $f$ has a $A_1^-$-singularity at the origin, then $c_{21}^2-4 c_{20} c_{22} > 0$. When $c_{22} \neq 0$ in a neighborhood of $P$, we have
	$$c_{20} c_{11}^2-c_{21} c_{10} c_{11}+c_{22} c_{10}^2 = 0 \quad \Leftrightarrow \quad c_{10} = \frac{c_{21} \pm \sqrt{c_{21}^2-4 c_{20} c_{22}}}{2 c_{22}} c_{11}.$$
	Therefore,
	$$\mathcal{I} = \left\{ c_{10} = \frac{c_{21} + \sqrt{c_{21}^2-4 c_{20} c_{22}}}{2 c_{22}} c_{11} \right\} \cup \left\{ c_{10} = \frac{c_{21} - \sqrt{c_{21}^2-4 c_{20} c_{22}}}{2 c_{22}} c_{11} \right\}.$$
	The cases $c_{20} \neq 0$ and $c_{20} = c_{22} = 0$ are analogous.
	
	For the vertices, fix $c_{20},c_{21},\cdots,c_{33}$ in a neighborhood of $P$ and consider the function
	$$h(c_{10},c_{11}) = v(c_{10},c_{11},\cdots,c_{33}).$$
	Note that
	\begin{equation}\label{eq:4-jet-g}
		j^4h = (c_{11}^2 c_{20} - c_{10} c_{11} c_{21} + c_{10}^2 c_{22}) (-2 c_{10} c_{11} c_{20} + c_{10}^2 c_{21} - c_{11}^2 c_{21} + 2 c_{10} c_{11} c_{22}).
	\end{equation}
	The function $h$ has a $X_9$-singularity, which is 4-$\mathcal{R}$-determined, when the discriminant of (\ref{eq:4-jet-g}) is non-zero \cite{ Arnold}, i.e.
	$$4 (c_{21}^2 - 4 c_{20} c_{22})^3 ((c_{20}-c_{22})^2 + c_{21}^2)^3 \neq 0.$$
	Therefore, if $f$ has a Morse singularity at the origin ($c_{21}^2 - 4 c_{20} c_{22} \neq 0$) and the origin is not an umbilic point of $C_f$ ($c_{20} \neq c_{22}$ or $c_{21} \neq 0$), then $h \sim_\mathcal{R} j^4h$ and $h=0$ is diffeomorphic to $j^4h=0$. 
	
	Note that
	$$c_{11}^2 c_{20} - c_{10} c_{11} c_{21} + c_{10}^2 c_{22} = 0$$
	it is composed of two transversal lines when $f$ has a $A_1^-$-singularity or just the origin ($c_{10} = c_{11} = 0$) when the singularity is $A_1^+$. On the other hand, the factor
	$$-2 c_{10} c_{11} c_{20} + c_{10}^2 c_{21} - c_{11}^2 c_{21} + 2 c_{10} c_{11} c_{22} = 0$$
	is always composed of two transversal lines at the origin when the origin is not a umbilic point. Therefore, when the origin is a singularity of $C_f$ and is not an umbilic point, $h = 0$ will have four (resp. two) branches when the singularity is of the type $A_1^-$ (resp. $A_1^+$) . The result follows by varying the parameters $c_{20},c_{21},\cdots,c_{33}$ that were fixed.
\end{proof}

%We will now proceed to examine the strata of inflection points and second-order vertices in the space $J^k(2,1)$.

Inflections of order 2 occur when $\kappa$ and $\kappa'$ vanishing simultaneously. Thus, the origin will be an inflection of order 2 of $C_f$ if and only if
\begin{equation}\label{eq:sis_i2_grande}
	\left\{ \begin{array}{l}
		c_{20} c_{11}^2-c_{21} c_{10} c_{11}+c_{22} c_{10}^2 = 0,\vspace{0.2cm}\\
		-2 c_{10} c_{11}^3 c_{20}^2 + 3 c_{10}^2 c_{11}^2 c_{20} c_{21} - c_{11}^4 c_{20} c_{21} -  c_{10}^3 c_{11} c_{21}^2+ c_{10} c_{11}^3 c_{21}^2\\
		- 2 c_{10}^3 c_{11} c_{20} c_{22} +  2 c_{10} c_{11}^3 c_{20} c_{22} + c_{10}^4 c_{21} c_{22} - 3 c_{10}^2 c_{11}^2 c_{21} c_{22} \\
		+  2 c_{10}^3 c_{11} c_{22}^2 + c_{10}^2 c_{11}^3 c_{30} + c_{11}^5 c_{30} - c_{10}^3 c_{11}^2 c_{31} -  c_{10} c_{11}^4 c_{31}\\
		+ c_{10}^4 c_{11} c_{32} + c_{10}^2 c_{11}^3 c_{32} - c_{10}^5 c_{33} -  c_{10}^3 c_{11}^2 c_{33} = 0.
	\end{array} \right.
\end{equation}
Using the first equation to simplify the second one, we see that the system (\ref{eq:sis_i2_grande}) is equivalent to
\begin{equation}\label{eq:sis_i2}
	\left\{ \begin{array}{l}
		c_{20} c_{11}^2-c_{21} c_{10} c_{11}+c_{22} c_{10}^2 = 0,\vspace{0.2cm}\\
		c_{11}^3 c_{30} - c_{10} c_{11}^2 c_{31} + c_{10}^2 c_{11} c_{32} -	c_{10}^3 c_{33} = 0.
	\end{array} \right.
\end{equation}
In this way, we define the submanifold of inflections of order 2 by
$$\begin{array}{ccl}
	\mathcal{I}_2 & = & \{(c_{00},c_{10}, \cdots, c_{kk}) \in J^k(2,1) : i(c_{10},c_{11},\cdots,c_{22}) =i_2(c_{10},c_{11},\cdots,c_{33}) = 0\}\vspace{0.2cm}\\
	& = & \{(c_{00},c_{10}, \cdots, c_{kk}) \in \mathcal{I} : i_2(c_{10},c_{11},\cdots,c_{33}) = 0\},
\end{array}$$
where
$$i_2(c_{10},c_{11},\cdots,c_{33}) = c_{11}^3 c_{30} - c_{10} c_{11}^2 c_{31} + c_{10}^2 c_{11} c_{32} -	c_{10}^3 c_{33}.$$

On the other hand, vertices of order 2 are points where $\kappa'$ and $\kappa''$ are zero. This will happen at the origin when
\begin{equation}\label{eq:sis_v2_grande}
	\left\{ \begin{array}{l}
		v(c_{10},c_{11},\cdots,c_{33}) = 0,\vspace{0.2cm}\\
		v_2(c_{10},c_{11},\cdots,c_{44}) = 0,		
	\end{array} \right.
\end{equation}
and the submanifold of vertices of order 2 is given by
$$\begin{array}{ccl}
	\mathcal{V}_2 & = & \{(c_{00},c_{10}, \cdots, c_{kk}) \in J^k(2,1) : v(c_{10},c_{11},\cdots,c_{22}) =v_2(c_{10},c_{11},\cdots,c_{33}) = 0\}\vspace{0.2cm}\\
	& = & \{(c_{00},c_{10}, \cdots, c_{kk}) \in \mathcal{V} : v_2(c_{10},c_{11},\cdots,c_{33}) = 0\},
\end{array}$$
with

$$\begin{array}{ccl}
	v_2(c_{10},c_{11},\cdots,c_{44}) & = & -4 c_{11}^6 c_{20}^3 - c_{10}^4 c_{11}^2 c_{20} c_{21}^2 + 4 c_{10}^2 c_{11}^4 c_{20} c_{21}^2 - 7 c_{11}^6 c_{20} c_{21}^2 + c_{10}^5 c_{11} c_{21}^3 \\ && + 7 c_{10} c_{11}^5 c_{21}^3 + 4 c_{10}^4 c_{11}^2 c_{20}^2 c_{22} + 4 c_{11}^6 c_{20}^2 c_{22} - 4 c_{10}^5 c_{11} c_{20} c_{21} c_{22} \\ && - 2 c_{10}^3 c_{11}^3 c_{20} c_{21} c_{22} + 10 c_{10} c_{11}^5 c_{20} c_{21} c_{22} - c_{10}^6 c_{21}^2 c_{22} - 6 c_{10}^4 c_{11}^2 c_{21}^2 c_{22} \\ && - 21 c_{10}^2 c_{11}^4 c_{21}^2 c_{22} + 4 c_{10}^6 c_{20} c_{22}^2 + 10 c_{10}^5 c_{11} c_{21} c_{22}^2 + 18 c_{10}^3 c_{11}^3 c_{21} c_{22}^2 \\ && - 4 c_{10}^6 c_{22}^3 - 4 c_{10}^4 c_{11}^2 c_{22}^3 - 10 c_{10}^3 c_{11}^4 c_{20} c_{30} - 10 c_{10} c_{11}^6 c_{20} c_{30} + 4 c_{10}^8 c_{44} \\ && + 11 c_{10}^4 c_{11}^3 c_{21} c_{30} + 18 c_{10}^2 c_{11}^5 c_{21} c_{30} + 7 c_{11}^7 c_{21} c_{30} - 12 c_{10}^5 c_{11}^2 c_{22} c_{30} \\ && - 16 c_{10}^3 c_{11}^4 c_{22} c_{30} - 4 c_{10} c_{11}^6 c_{22} c_{30} + 6 c_{10}^4 c_{11}^3 c_{20} c_{31} + 2 c_{10}^2 c_{11}^5 c_{20} c_{31} \\ && - 4 c_{11}^7 c_{20} c_{31} - 7 c_{10}^5 c_{11}^2 c_{21} c_{31} - 10 c_{10}^3 c_{11}^4 c_{21} c_{31} - 3 c_{10} c_{11}^6 c_{21} c_{31} \\ && + 8 c_{10}^6 c_{11} c_{22} c_{31} + 8 c_{10}^4 c_{11}^3 c_{22} c_{31} - 2 c_{10}^5 c_{11}^2 c_{20} c_{32} + 6 c_{10}^3 c_{11}^4 c_{20} c_{32} \\ && + 8 c_{10} c_{11}^6 c_{20} c_{32} + 3 c_{10}^6 c_{11} c_{21} c_{32} + 2 c_{10}^4 c_{11}^3 c_{21} c_{32} - c_{10}^2 c_{11}^5 c_{21} c_{32} \\ && - 4 c_{10}^7 c_{22} c_{32} + 4 c_{10}^3 c_{11}^4 c_{22} c_{32} - 2 c_{10}^6 c_{11} c_{20} c_{33} - 14 c_{10}^4 c_{11}^3 c_{20} c_{33} \\ && - 12 c_{10}^2 c_{11}^5 c_{20} c_{33} + c_{10}^7 c_{21} c_{33} + 6 c_{10}^5 c_{11}^2 c_{21} c_{33} + 5 c_{10}^3 c_{11}^4 c_{21} c_{33} \\ && - 8 c_{10}^6 c_{11} c_{22} c_{33} - 8 c_{10}^4 c_{11}^3 c_{22} c_{33} + 4 c_{10}^4 c_{11}^4 c_{40} + 8 c_{10}^2 c_{11}^6 c_{40} + 4 c_{11}^8 c_{40} \\ && - 4 c_{10}^5 c_{11}^3 c_{41} - 8 c_{10}^3 c_{11}^5 c_{41} - 4 c_{10} c_{11}^7 c_{41} + 4 c_{10}^6 c_{11}^2 c_{42} + 8 c_{10}^4 c_{11}^4 c_{42} \\ && + 4 c_{10}^2 c_{11}^6 c_{42} - 4 c_{10}^7 c_{11} c_{43} - 8 c_{10}^5 c_{11}^3 c_{43} - 4 c_{10}^3 c_{11}^5 c_{43} + 8 c_{10}^6 c_{11}^2 c_{44} \\ && + 4 c_{10}^4 c_{11}^4 c_{44}.\\
\end{array}$$
\begin{rem}
	In a similar way, we can define the submanifolds of inflections and vertices of order $n \geq 3$, denoted by $\mathcal{I}_n$ and $\mathcal{V}_n$, such that
	$$\mathcal{I} \supset \mathcal{I}_2 \supset \mathcal{I}_3 \supset \cdots \quad {\rm and} \quad \mathcal{V} \supset \mathcal{V}_2 \supset \mathcal{V}_3 \supset \cdots.$$
\end{rem}

\subsection{Codimension}

Let $f : (\R^2,0) \to \R$ be a smooth function germ. A \textit{deformation of $f$ with $n$ parameters}, or a \textit{$n$ parameters family of $f$}, is a smooth function germ $F:(\R^2 \times \R^n,0) \to \R$ with
$$F(x,y,0,\cdots,0)=f(x,y).$$
Consider
$$\mathcal{F} = \{F:(\R^2 \times \R^n,0) \to \R \, ; \, F(x,y,0,\cdots,0)=f(x,y)\}$$
with the Whitney topology. A deformation or family is said to be \textit{generic} when it belongs to an open and dense subset of $\mathcal{F}$.

A point at $C_f$ with a property ($P$) has \textit{codimension} $n$ when it occurs generically in $n$-parameters families of curves.

\begin{prop}
	A point at $C_f$ with a property ($P$) has \textit{codimension} $n$ if and only if the submanifold in $J^k(2, 1)$, for $k$ sufficiently large, obtained by the conditions imposed by ($P$) on the coefficients of the germ $f$ has codimension $n + 1$.
\end{prop}

\begin{proof}
	The proof of the proposition follows from the transversality of the Monge-Taylor map.
\end{proof}

Note that singular points of $C_f$ have codimension greater than or equal to 1. In fact, $(x_0,y_0)$ is a singular point of $f$ if and only if $\Phi_f(x_0,y_0) \in \mathcal{S}$ and $\mathcal{S}$ has codimension 2. In this work, we will study singularity points with codimension less than or equal to 2. The points of $C_f$ with codimension 1 are defined by open submanifolds of $\mathcal{S}$, since these submanifolds also have codimension 2 in $J^k(2,1)$. The possible cases are:
\begin{itemize}
	\item $A_1^-$ singularities whose branches have no vertices and inflections at origin,
	\item $A_1^+$ singularities that are not umbilic at origin.
\end{itemize}
On the other hand, the cases of codimension 2 are obtained by imposing one, and only one, additional condition on the singularity in order to increase the codimension of the submanifold of $\mathcal{S}$. The points of $C_f$ with codimension 2 are:
\begin{itemize}
	\item $A_1^-$ singularities whose one of the branches has an ordinary inflection, has no vertice and the other branch has no vertice and inflextion at the origin, 
	\item $A_1^-$ singularities whose one of the branches has an ordinary vertex, has no inflextion and the other branch has no vertice and inflextion at the origin,
	\item $A_2$ singularities.
\end{itemize}

\section{FRS-Equivalence}
Diffeomorphisms do not preserve the geometry of curves such as inflections and vertices. Therefore, we cannot use the group $\mathcal{R}$ or $\mathcal{A}$ to study deformations when we are interested in these geometric properties of the curve $f=0$. Thus, we describe a method, similar to the one introduced in \cite{FRS} for parameterized curves, to study the geometry of deformations of singular plane curves $f=0$ with $f\in\mathcal{M}_2$, called \textit{FRS-deformations of implicit plane curves}. %When the curve is regular, the method is called FR-deformations.

Consider two germs of $m$ parameter deformations $F_s$ and $\tilde{F}_u$ of the same plane curve $f=0$, whit $s, u\in (\mathbb{R}^m,0)$. Take a stratification germ $(S_1,0)$ of $\mathbb{R}^m$ such that if
$s_1$ and $s_2$ are in the same stratum, then the curves $F_{s_1}=0$ and $F_{s_2}=0$ satisfy the following properties:
\begin{itemize}
	\item[(i)] they are diffeomorphic;
	\item[(ii)] they have the same numbers of inflections and vertices;
	\item[(iii)] they have the same relative position of their singularities, points of self-intersection, inflections and vertices.
\end{itemize}

Let $(S_2,0)$ be another stratification of $\mathbb{R}^m$ satisfying (i), (ii) and (iii) for $\tilde{F}$. We say that these deformations are FRS-equivalent if there is a stratified homeomorphism $k:(\mathbb{R}^m,(S_1,0)) \to (\mathbb{R}^m,(S_2,0))$ such that all pairs of curves $F_s=0$ and $\tilde{F}_{k(s)}=0$, with $s \in (\mathbb{R}^m,0)$, satisfy properties (i), (ii) and (iii).

Let $F_s$ and $\tilde{F}_u$ be deformations of $f$ with $m$ and $n$ parameters ($m \leq n$). If $s \in \R^m$ and $t \in \R^{n-m}$, then $(s,t) \in \R^n$. Thus, $F_{(s,t)} = F_s$ is a deformation of $f$ with $n$ parameters. We say that $F_s$ and $\tilde{F}_u$ are FRS-equivalent when $F_{(s,t)}$ and $\tilde{F}_u$ are FRS-equivalent.
We aim to classify deformations of singularities up to FRS-equivalence, focusing on those of codimension 2 or less.

\section{$A_1^-$-singularity}\label{sec:A1-}

Let $f:(\R^2,0) \to \R$ be a smooth function germ. If the origin is a $A_1^-$-singularity of $f$ and belongs to the curve $C_f$, then
$$c_{00} = c_{10} = c_{11} = 0 \quad {\rm and} \quad c_{21}^2-4c_{20} c_{22} > 0.$$
Using isometries and homotheties, which preserve vertices and inflections, we can consider $c_{20} = 0$ and $c_{21} = 1$, that is
\begin{equation}\label{eq:fA1-simplif}
	j^kf=x y+a_{22} y^2+\sum_{i=3}^{k} \sum_{j=0}^i \q{a}{ij}x^{i-j}y^j.
\end{equation}

Assuming \eqref{eq:fA1-simplif}, a branche of $C_f$ is parallel to the $x$-axis and the other is transversal. Denote these branches by $\alpha_1$ and $\alpha_2$, respectively. Locally at the origin, we can parameterize them by $\alpha_1(s)=(s,y(s))$ and $\alpha_2(s)=(x(s),s)$. Since $f \circ \alpha_i  \equiv 0$, deriving it implicitly follows that
$$
j^3\alpha_1=(s, -\eta_1 s^2 - \eta_3 s^3 ) \ \ \text{and}\ \
j^3\alpha_2=(-\add s +\eta_{2} s^2-\bar{\eta}_{4} s^3,s),
$$
where 
	$$\begin{array}{ccl}
	\eta_1 & = & a_{30},\\
	\eta_2 & = & a_{22}^3 a_{30}-a_{22}^2 a_{31}+a_{22} a_{32}-a_{33}, \text{ and }\\
	\eta_3 & = & a_{22} a_{30}^2-a_{30} a_{31}+a_{40},\\
	\bar{\eta}_4 & = & 3 a_{22}^5 a_{30}^2 - 5 a_{22}^4 a_{30} a_{31} + 2 a_{22}^3 a_{31}^2 + 4 a_{22}^3 a_{30} a_{32} - 3 a_{22}^2 a_{31} a_{32} + a_{22} a_{32}^2 \\& &- 3 a_{22}^2 a_{30} a_{33} + 2 a_{22} a_{31} a_{33} - 
	a_{32} a_{33} + a_{22}^4 a_{40} - a_{22}^3 a_{41} + a_{22}^2 a_{42} - a_{22} a_{43} + a_{44}.	
\end{array}$$

\begin{prop}\label{prop:conditions}
	The following statements hold true.
	\begin{description}
		\item[a)] $\alpha_1$ has a first-order inflection at the origin if and only if $\eta_1=0$ and $a_{40}\neq0$.
		\item[b)] $\alpha_1$ has a first-order vertex at the origin if and only if $\eta_3=0$ and $\atz^3 - \atz^2 \atd + \atz a_{41} - a_{50}\neq0$.
		\item[c)] $\alpha_1$ has a second-order inflection at the origin if and only if $\eta_1=a_{40}=0$ and $\q{a}{50}\neq0$.
		\item[d)] $\alpha_1$ has a second-order vertex at the origin if and only if $\eta_3=\atz^3 - \atz^2 \atd + \atz a_{41} - a_{50}=0$ and $-a_{30} a_{31} a_{32}^2 + a_{30}^2 a_{32} a_{33} + a_{30} a_{32} a_{40} + a_{32}^2 a_{40} + 
		a_{30} a_{31} a_{41} + a_{22} a_{30} a_{32} a_{41} - a_{30} a_{33} a_{41} - 2 a_{40} a_{41} - 
		a_{30}^2 a_{42} - 2 a_{22} a_{30} a_{50} + a_{31} a_{50} - a_{22} a_{32} a_{50} + a_{33} a_{50} + a_{30} a_{51}-a_{60}\neq0$.
	\end{description}
\end{prop}
\begin{proof}
The result follows from the 3-jets given above for the curves $\alpha_1$ and $\alpha_2$ and direct calculations.
\end{proof}

\begin{rem}
	There are similar conditions for $\alpha_2$, it will be omitted because they are too lengthy.
\end{rem}
For the following, we state
\begin{equation}
	\begin{array}{ccl}
		\eta_4 & = & 3 a_{22}^5 a_{30}^2 + a_{22}^7 a_{30}^2 - 5 a_{22}^4 a_{30} a_{31} - a_{22}^6 a_{30} a_{31} + 
		2 a_{22}^3 a_{31}^2  + a_{22}^4 a_{40} + a_{22}^6 a_{40} \\
		& &  - 3 a_{22}^2 a_{31} a_{32} + a_{22}^4 a_{31} a_{32} + a_{22} a_{32}^2 - a_{22}^3 a_{32}^2 - 3 a_{22}^2 a_{30} a_{33} + a_{22}^4 a_{30} a_{33}\\
		& &  + 2 a_{22} a_{31} a_{33} - 2 a_{22}^3 a_{31} a_{33} - a_{32} a_{33} + 3 a_{22}^2 a_{32} a_{33} -2 a_{22} a_{33}^2  + 4 a_{22}^3 a_{30} a_{32} \\
		& &  - a_{22}^3 a_{41} - a_{22}^5 a_{41} + 
		a_{22}^2 a_{42} + a_{22}^4 a_{42} - a_{22} a_{43}- a_{22}^3 a_{43} + a_{44} + a_{22}^2 a_{44}.
	\end{array}
\end{equation}

\begin{prop}\label{prop:I2-V2}
	Let $f:(\mathbb{R}^2,0) \to (\mathbb{R},0)$ be a germ of smooth satisfying \eqref{eq:fA1-simplif}.
	In a neighborhood of $P = \Phi_f(0,0)$, the submanifold $\mathcal{I}_2$ (resp. $\mathcal{V}_2$) coincides with $\mathcal{S}$ when $\eta_1 \eta_2 \neq 0$ (resp. $\eta_3 \eta_4 \neq 0$).

	In the case where $\eta_1=0$ or $\eta_2=0$ (resp. $\eta_3=0$ or $\eta_4=0$), there is a regular connected component of $\mathcal{I}_2$ (resp . $\mathcal{V}_2$) with codimension two and passing at $P$ for each $\eta_i$ vanishing, besides $\mathcal{S}$.
\end{prop}

\begin{proof}
%		Consider $i$, $i_2$, $v$ and $v_2$ functions in the variables $c_{10}$ and $c_{11}$, with \linebreak $A= (c_{20},c_{21},\cdots,c_{44})$ fixed. Let $ A_P = (0,1,a_{22},a_{30},\cdots, a_{44})$, that is, $A_P$ is the value of $A$ at $P$. {\color{red} não entendi, não é o valor de $\Phi_f$ na origem? mas sem as 3 primeiras coordenadas no caso aí, pq elas tão variando, né?}

%When $A=A_P$, the resultant between the polynomials $i$ and $i_2$ with respect to $c_{10}$ is $c_{11}^6 \eta_1 \eta_2$. Thus, if $\eta_1 \eta_2 \neq 0$, then the solutions of (\ref{eq:sis_i2}) occur when $c_{11} = 0$ and we conclude that $c_{10} = 0$ and $\mathcal{I}_2$ coincides with $\mathcal{S}$ in a neighborhood of $P$.

	Consider $i$, $i_2$, $v$ and $v_2$ functions in the variables $c_{10}$ and $c_{11}$, with $c_{20},c_{21},\cdots,c_{44}$ fix. Let $A=(c_{20},c_{21},\cdots,c_{44})$ the point formed by the fixed coordinates. Define $A_P$ so that $P=(0,0,0,A_P)$, that is, $A_P=(0,1,a_{22},a_{30},\cdots, a_{44})$.
	
	When $A=A_P$, the resultant between the polynomials $i$ and $i_2$ with respect to $c_{10}$ is $c_{11}^6 \eta_1 \eta_2$. Thus, if $\eta_1 \eta_2 \neq 0$, then the solutions of (\ref{eq:sis_i2}) occur when $c_{11} = 0$ and we conclude that $c_{10} = 0$ and $\mathcal{I}_2$ coincides with $\mathcal{S}$ in a neighborhood of $P$.
	
	Regarding vertices, it follows from the proof of Proposition \ref{prop:IV} that $v = 0$ is formed by 4 branches that are two by two transversal at the origin. On the other hand, $v_2 = 0$ has between 3 and 7 branches that are also two by two transversal at the origin. In fact, we have that
	$$j^7v_2 = (c_{11}^2 c_{20}-c_{10} c_{11} c_{21}+c_{10}^2 c_{22}) h(c_{10},c_{11}),$$
	where $h$ is a homogeneous polynomial of degree 5 in $c_{10}$ and $c_{11}$ with no factors in common with $c_{11}^2 c_{20}-c_{10} c_{11} c_{21}+c_{10}^2 c_{22}$ when $A$ is close to $A_P$.
	
	Since $c_{11}^2 c_{20}-c_{10} c_{11} c_{21}+c_{10}^2 c_{22}$ is a common factor of the tangent cone of $v = 0$ and $v_2 = 0$, there are branches $\delta_i$ of $v=0$ and $\bar{\delta}_i$ of $v_2=0$ such that $\delta_i$ and $\bar{\delta}_i$ are tangent at the origin, for $i= 1, 2$. Note that the tangent cones of $\delta_i$ and $\bar{\delta}_i$, for $i=1, 2$, are the linear factors of $c_{11}^2 c_{20}-c_{10} c_{11} c_{21}+c_{10}^2 c_{22}$. Thus, suppose that the tangent cone of $\delta_1$ and $\bar{\delta}_1$ (resp. $\delta_2$ and $\bar{\delta}_2$) is
	$$c_{10} = \frac{c_{21} - \sqrt{c_{21}^2-4c_{20}c_{22}}}{2 c_{22}} c_{11} \quad \left({\rm resp. } \,\,\, c_{10} = \frac{c_{21} + \sqrt{c_{21}^2-4c_{20}c_{22}}}{2 c_{22}} c_{11}\right).$$
	
	When $A=A_P$, using the Implicit Function Theorem we can parameterize $\delta_1$ by $(g(t),t)$ and $\bar{\delta}_1$ by $(\bar g(t),t)$. By implicitly deriving it, it is possible to calculate $j^3 g$ and $j^3 \bar g$ and verify that
	$$j^3 g-j^3 \bar g = \eta_3 t^3.$$
	Therefore, if $\eta_3 \neq 0$, then the contact order between $\delta_1$ and $\bar{\delta}_1$ is three for values of $A$ in a neighborhood of $A_P$. Likewise, we see that the contact order between $\delta_2$ and $\bar{\delta}_2$ is also three for values of $A$ in a neighborhood of $A_P$ when $\eta_4 \neq 0$.
	
	Therefore, when $\eta_3 \eta_4 \neq 0$, the contact order between $\delta_i$ and $\bar\delta_i$ remains unchanged in a neighborhood of $A_P$. Thus, for any parameter values $c_{20},c_{21},\cdots,c_{44}$ close to $A_P$, the curves $v=0$ and $v_2=0$ intersect only at the origin. In this way, $\mathcal{V}_2$ coincides with $\mathcal{S}$.
	
	Suppose that $\eta_3 = 0$. The contact order between $\gamma_1$ and $\bar \gamma_1$ at the origin is, generically, 4 at $A_P$ and 3 for values of $A$ close to $A_P$. It follows from the transition between contact order 3 and 4 (see Figure \ref{fig:transition_3_for_4}), that for each $c_{20},c_{21},\cdots,c_{44}$ in a neighborhood of $A_P$ exists a point $Q(c_{20},\cdots,c_{44})$ at the intersection of $\delta_i$ and $\bar\delta_i$ such that $\delta_i$ and $\bar\delta_i$ are transversals in $Q(c_{20},\cdots,c_{44})$ and the limit of $Q(c_{20},\cdots,c_{44})$ as $(c_{20},\cdots,c_{44})$ tends to $A_P$ is $P$. By varying the parameters $c_{20},c_{21},\cdots,c_{44}$ we obtain the desired submanifold. The result is analogous when $\eta_1$, $\eta_2$ and $\eta_4$ are different to 0.
	\begin{figure}[h!]
		\centering
		\includegraphics[width=\textwidth]{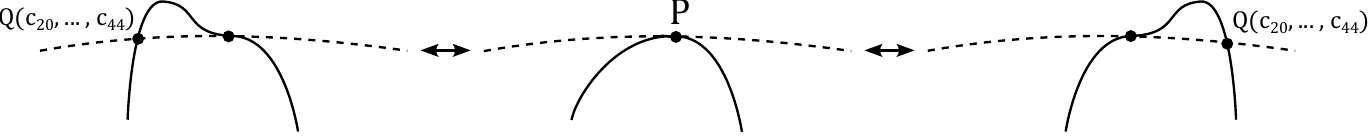}
		\caption{Transition between contact order 3 and 4. The solid curve is $\delta_1$ and the dashed one is $\bar\delta_1$.}
		\label{fig:transition_3_for_4}
	\end{figure}
\end{proof}

A deformation of $f$ with $p$ parameters is given by
\begin{eqnarray}\label{eq:Fcod2i}
	F(x,y,\textbf{u})=\sum_{i+j=0}^{k}\q{A}{ij}(\textbf{u})x^{i-j}y^j +O(k+1),
\end{eqnarray}
with $\textbf{u}=(u_1,\cdots,u_p) \in \mathbb{R}^p$, $\q{A}{ij}(0)=\q{a}{ij}$, and $O(k+1)$ is the Lagrange remainder of $F$.
For $F$ to be {$\mathcal{R}_e$-versal}, it is necessary that $\dfrac{\partial A_{00}}{\partial u_i}(u_i)\neq0$ for some $i=1,...,n$.
%Indeed, we have $L\mathcal{R}\cdot f=\{x,y\}$ and, thus, the complement of the tangent space at $\mathcal{E}_2$ is generated by $\mathbb{R}\cdot\{1\}$.
Suppose, without loss of generality, that $A_{00}(\textbf{u})=u_1$. Recall that a versal deformation of a $A_1^-$-singularity is as in Figure \ref{fig:deformacaogenerica}.
\begin{figure}[h!]
	\centering
	\includegraphics[width=7cm]{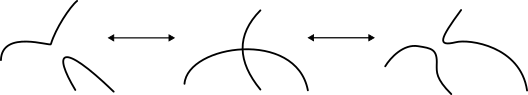}
	\caption{Versal deformation of a $A_1^-$ singularity.}%, when the helicoidal surface is not a revolution surface.}
\label{fig:deformacaogenerica}
\end{figure}

It follows from Proposition \ref{prop:IV} that the $\mathcal{I}$ (resp. $\mathcal{V}$) is formed by two (resp. four) submanifolds with codimension 1 whose intersection is transversal and coincides with the singular stratum $\mathcal{S}$. Since $F$ is versal, the Monge Taylor map $\Phi_F$ is transversal to $\mathcal{S}$ and, consequently, transversal to $\mathcal{I}$ and $\mathcal{V}$. Therefore, $\Phi_F^{-1}(\mathcal{I})$ (resp. $\Phi_F^{-1}(\mathcal{V})$) are two (resp. four) regular manifolds of codimension %$n-2$ in $\R^n$, $n\geq3$
1 in $\R^{2+p}$. Clearly, their intersections with $F^{-1}(0)$ near the origin result in manifolds of codimension $2$.
%, i.e., regular curves in $\mathbb{R}^3$. 
These manifolds are solutions of $F=I=0$ for the inflections and $F=V=0$ for the vertices, where $I(x,y,\textbf{u})=i_F(x,y)$ and $V(x,y,\textbf{u})=v_F(x,y)$. 

\subsection{Case of codimension 1}

Let $F$ be a versal deformation of $f$ with one parameter $t$ as (\ref{eq:Fcod2i}). Since $F_t(0,0,0)=1$, by Implicit Function Theorem, there exists a smooth function $T(x,y)$ such that $$F(x,y,T(x,y))=0,$$ for any $(x,y)$ near the arigin. Thus, $(x,y,t)$ is a solution of $F = I = 0$ if and only if $t = T(x,y)$ and $\bar{I}(x,y)=I(x,y,T(x,y)) = 0$. Direct calculations reveal that
$$
j^2 \bar{I}=2 j^2 f= 2y(\add y + x).
$$
Since $\bar I = 0$ has two branches, these branches are tangent to the branches of $C_f$ and can be parameterized by $\beta_1(s)=(s,b_1(s))$ and $\beta_2(s)=(b_2(s),s)$ with

$$
\begin{array}{ccl}
	b_1(s) & = & (9 \add \atz^2 - 3 \atz \atu + 2 a_{40} + \atz A_{10}'(0))s^3+O(4),\vspace{0.2cm}\\
	
	b_2(s)& = & -\add s
	-(-9 \add^5 \atz^2 + 15 \add^4 \atz \atu - 6 \add^3 \atu^2 - 
	12 \add^3 \atz \atd + 9 \add^2 \atu \atd \\ &&- 3 \add \atd^2 + 9 \add^2 \atz \att- 
	6 \add \atu \att + 3 \atd \att - 2 \add^4 a_{40} + 2 \add^3 a_{41}- 
	2 \add^2 a_{42}  \\ && + 2 \add a_{43}-2a_{44}-\add^4 \atz A_{10}'(0) + 
	\add^3 \atu A_{10}'(0) - \add^2 \atd A_{10}'(0) + 
	\add \att  A_{10}'(0)  \\ && + \add^3 \atz  A_{11}'(0) - 
	\add^2 \atu  A_{11}'(0) + \add\atd  A_{11}'(0) - \att  A_{11}'(0)) s^3+O(4),
\end{array}
$$
where $O(4)$ are terms with order greater than or equal to 4 at $s$. The curves $\beta_i$ are called \textit{inflection curves}.

Similarly for the vertices, we define $\bar{V}(x,y)=V(x,y,T(x,y))$. Since $\bar V(x,y)=0$ has 4 branches and
$$
j^4 \bar{V}=-6 y (x + \add y) (x^2 + 2 \add x y - y^2),
$$
it follows that two branches of $\bar{V}=0$ are tangent to the branches of $C_f$, and the other two are always transverse to them. The branches of $\bar V = 0$ tangent to $C_f$ can be parameterized by
$$
\begin{array}{ccl}
	\gamma_1(s) & = & (s,-\eta_1 s^2-2\eta_3s^3+O(3)),\\
	\gamma_2(s) & = & (-\add s+\eta_2s^2+O(3),s),
\end{array}
$$
and the other branches are given by
$$
\begin{array}{ccl}
	\gamma_3(s) & = & (s,\add+\sqrt{1+\add^2} \, s+O(2))\vspace{0.2cm}\\
	\gamma_4(s) & = & (s,\add-\sqrt{1+\add^2} \, s+O(2)),
\end{array}
$$
where $O(n)$ are terms with order greater than or equal to $n$ at $s$. The curves $\gamma_i$ are called \textit{vertex curves}.

%Let $f:(\R^2,0) \to \R$ be a germ with an $A_1^-$ singularity at the origin. Morse singularities have codimension one. 
In this section, take $\eta_i\neq0$ for $i=1,2,3$ and $4$. Thus, the branches of $C_f$ has no vertices or inflections at the origin (by Proposition \ref{prop:conditions}). In this case, we say that the origin is a \textit{generic} $A_1^-$ singularity.

Considering the curves inflections and vertex curves, we note that
$$
\begin{array}{ccl}
	\alpha_1(s)-\beta_1(s) & = & (0,-\eta_1 s^2+O(3)),\\
	\alpha_1(s)-\gamma_1(s) & = & (0,\eta_3s^3+O(4)),\\
	\alpha_2(s)-\beta_2(s) & = & (\eta_2s^2+O(3),0),\\
	\alpha_2(s)-\gamma_2(s) & = & (\eta_4s^3+O(4),0)
\end{array}
$$
and, since $\eta_1, \eta_3, \eta_2$, and $\eta_4$ are non-zero, the tangency between the branches of $C_f$ and the inflection curves will be ordinary. On the other hand, with the vertex curves, the tangencies are of order two.

%{\color{red} By Proposition \ref{prop:I2-V2}, the stratum $\mathcal{S}$ locally coincides with the strata $\mathcal{I}_2$ and $\mathcal{V}_2$. Therefore, $\Phi_f$ being transversal to $\mathcal{S}$ is equivalent to it being transversal to $\mathcal{I}_2$ and $\mathcal{V}_2$.

%For this reason, we say that a one-parameter deformation of a generic singularity $A_1^-$ is \textit{FRS-versal} if it is $\mathcal{A}_e$-versal. ISSO PODE SER QUE SAIA}

%For fixed $t$, let $f_t(x,y) = F(x,y,t)$. Thus, $C_{f_t}$ is a deformation of $C_f$. The inflections and vertices of $C_{f_t}$ are given by the intersections of $f_t(x,y)=0$ with the curves $\beta_i$ and $\gamma_i$, respectively. So, to find how many inflections and vertices appear on $C_{f_t}$, with $t$ near the origin, is enough to find how many inflection and vertices curves are in each region delimited by $C_f$ and their arrangement on it.

\begin{defn}
	To describe the geometric configuration of a one-parameter versal deformation $F$ of $f$, we will write, for example, $vivi-vv \leftrightarrow viv-vvi$, where each side of the arrow represents a sign of the parameter, $i$ and $v$ represent inflections and vertices, and each subpart separated by the hyphen represents one of the connected parts of the deformation curve. Furthermore, the order of the letters indicates the sequence in which the inflections and vertices appear.
\end{defn}

The inflection branches are tangent to the branches of \( C_f \) with second-order contact. This implies that the transition of inflections will always occur as \( ii-\varnothing \leftrightarrow i-i \). The sign of the parameter on each side of the notation \( ii-\varnothing \leftrightarrow i-i \) will be significant; to determine this, we calculate \( T(\beta_i(s)), i=1,2 \), given by \( T(\beta_1(s)) = -\eta_1 s^3 + \cdots \) and \( T(\beta_2(s)) = 24 \eta_2 s^3 + \cdots \). Thus, if \( \eta_1 > 0 \), for example, the region bounded by \( C_f \) containing the part of the curve \( \beta_1 \) with \( s > 0 \) corresponds to the parameter \( t < 0 \). By applying the same analysis to all combinations of signs between \( \eta_1 \) and \( \eta_2 \), we conclude that \( \eta_1 \eta_2 > 0 \) if and only if the regions with two inflections are represented by deformations with a positive parameter. Figure \ref{fig:exemplocurvasinflexoes} illustrates the example where \( \eta_1 \eta_2 > 0 \).

\begin{figure}[h!]
	\centering
	\includegraphics[width=5cm]{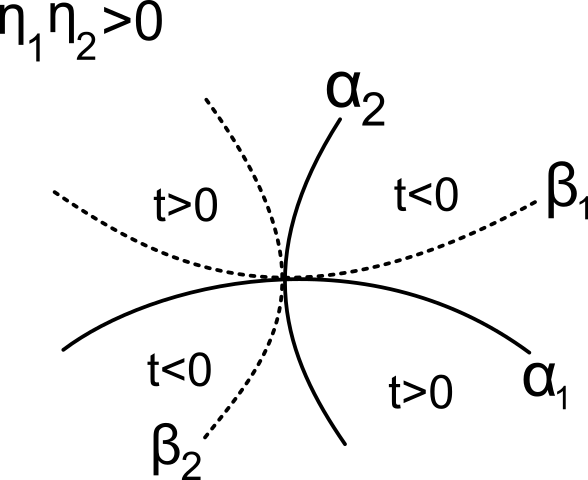}
	\caption{Exemple of the configuration of $C_f$ and the inflexion curves, when $\eta_1\eta_2>0$.}
\label{fig:exemplocurvasinflexoes}
\end{figure}

For vertex curves, we note that, since we always have two branches of vertices transverse to the branches of $C_f$ and their respective tangent lines have opposite angular coefficients, there will always be a vertex branch in each region and this will be at the middle of the other branches, if they appear.

Now we consider the other vertex curves. Suppose $\eta_3 \eta_4<0$. Note that, in this case, the coefficient of third degree (the first non-zero one) of $\alpha_i(s)-\gamma_i(s)$ has distinct signs for different $i$. This implies that there is one more vertex in each region of $f=0$. As the contact order of the branches of $C_f$ with the branches of vertices is always greater then with the inflections, we have that the only possible configuration in this case is given in Figure \ref{fig:exemplo1}, for the case $\eta_1\eta_2>0$ and the geometry of the deformation is showed in \ref{fig:casoscod=1}, third subfigure.
\begin{figure}[h!]
	\centering
	\includegraphics[width=6cm]{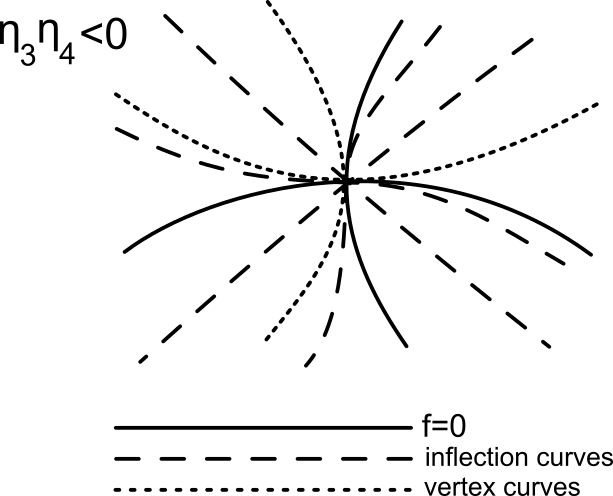}
	\caption{Configuration of the vertex and inflexional curves for $\eta_1\eta_2>0$ and $\eta_3\eta_4<0$.}
	\label{fig:exemplo1}
\end{figure}

\begin{figure}[h!]
	\centering
	\includegraphics[width=8cm]{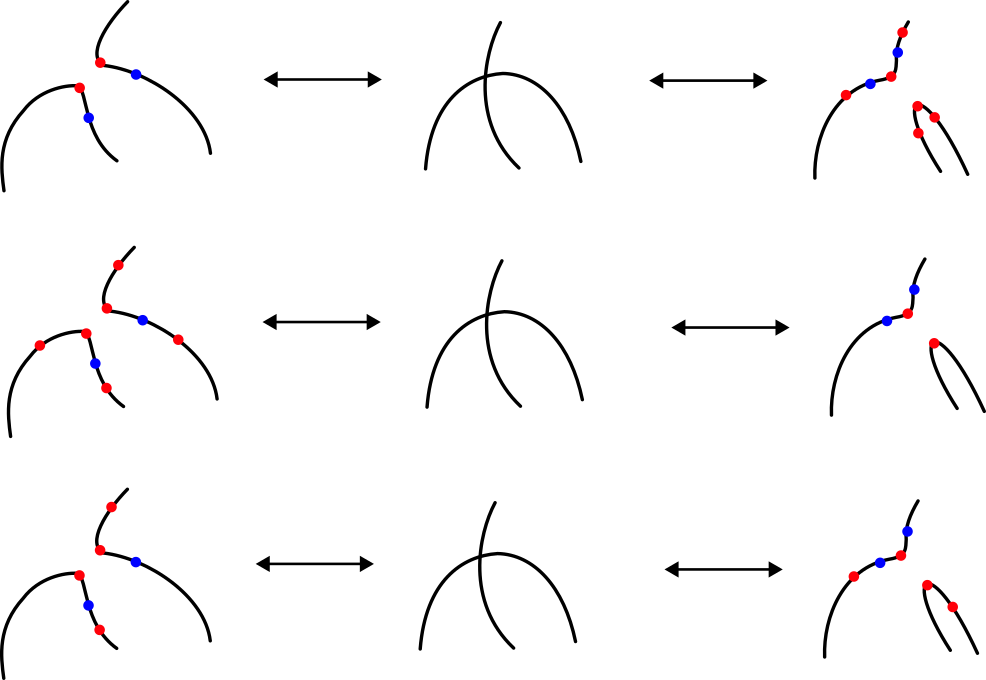}
	\caption{Possible geometric configuration for the general case (codimension 1). Red cdots are vertices and blue cdots are inflections.}
	\label{fig:casoscod=1}
\end{figure}

Now suppose $\eta_3 \eta_4>0$. In this case, with the same analysis as in the previous case, we conclude that there are two vertices in two opposite regions of $C_f$ and none in the other regions. Direct calculations show that $T(\gamma_1)=\eta_3 s^4 +\cdots$. Hence, it holds that $\eta_3>0$ if and only if the regions with two vertices are those represented by positive parameter.

Thus, we conclude that if $\eta_1\eta_2\eta_3>0$, then the regions with two additional vertices coincide with those where the inflections are of the $ii-\varnothing$ form. Again, considering that the tangencies of $C_f$ with the branches of vertices have of higher order than the tangencies with the branches of inflections, we have that the only possible behavior in this case is given in Figure \ref{fig:casoscod=1}, first subfigure. Now, for the case $\eta_1\eta_2\eta_3<0$, the reasoning follows analogously.

Since we only used the hypothesis of $F$ being \(\mathcal{A}\)-versal, we have just proved the following result.

%{\color{green} Finally, the solutions of $F = I = 0$ are the curves $\bar{\beta}_i(s)=(\beta_i(s),T(\beta_i(s)))$ in $\R^3$ for $i=1,2$. Likewise, the solutions of $F = V = 0$ are the curves $\bar \gamma_i(s) = \left( \gamma_i(s),T(\gamma_i(s)) \right)$ in $\R^3$ for $i=1,\cdots,4$. SE PERDEU, PRECISO ACHAR ONDE}

\begin{theo}\label{teo:geocod2casei}
	Let $f:(\mathbb{R}^2,(0,0))\to(\mathbb{R},0)$ be a smooth function germ with a generic $A_1^-$ singularity. Then, any 1-parameter $\mathcal{A}$-versal deformation of $f$ is $FRS$-equivalent to $F:(\mathbb{R}^3,0)\to(\mathbb{R},0)$ with $F(x,y,t)=f(x,y)+t$. We have the following possible geometric configurations for $F$:
	\begin{itemize}
		\item $vivi - vv \leftrightarrow viv - vvi$, if $\eta_3\eta_4<0$.
		\item $vi - vi \leftrightarrow viviv - vvv$, if $\eta_3\eta_4>0$ and $\eta_1 \eta_2\eta_3>0$.
		\item $ivi - v \leftrightarrow vviv - vviv$, if $\eta_3\eta_4>0$ and $\eta_1 \eta_2\eta_3<0$.		
	\end{itemize}
\end{theo}

%\begin{theo}\label{teo:geocod2casei}
%Let $f:(\mathbb{R}^2,(0,0))\to(\mathbb{R},0)$ be a smooth function germ with $A_1^-$ singularity. Then, any 1-parameter $\mathcal{A}$-versal deformation of $f$ is $FRS$-equivalent to $F:(\mathbb{R}^3,0)\to(\mathbb{R},0)$ with $F(x,y,t)=f(x,y)+t$. We have the following possible geometric configurations for $F$:
%\begin{itemize}
%	\item vivi - vv // viv - vvi, if $\eta_3\eta_4<0$.
%	\item vi - vi // viviv - vvv, if $\eta_3\eta_4>0$ and $\eta_1 \eta_2\eta_3>0$.
%	\item ivi - v // vviv - vviv, if $\eta_3\eta_4>0$ and $\eta_1 \eta_2\eta_3<0$.		
%\end{itemize}
%\end{theo}
%
%\begin{proof}
%	We know that the deformation curve $F_t=0$ exhibit an inflection or a vertex if it intersect one of the curves $\beta_i$ or $\gamma_i$, $i=1,2$ respectively. 
%\end{proof}

\subsection{Cases of codimension 2}

In this section, we will consider two geometric versal deformations of a $A_1^-$-singularity with its $2$-jet as in \eqref{eq:fA1-simplif} and with geometric codimension two. These are the cases where only one of the branches of $C_f$ exhibits, at the origin, an inflection or a vertex of first order.

\subsubsection{One of the branches of $C_f$ has an inflection at the origin}

Without loss of generality, let us suppose that $\alpha_1$ is the curve with an inflection of order one at the origin and $\alpha_2$ has neither a vertex nor an inflection at the origin. This is equivalent to stating that in \eqref{eq:fA1-naosimplif}, $\eta_1=\atz=0$ and $\eta_2$, $\eta_3$, and $\eta_4$ are non-zero.
% {\color{blue}By Definition ??, for $F$ to be a geometric deformation, it is necessary for $\Phi$ to be transversal to the $x_{7}=0(\atz=0)$ stratum in the jet space, in addition to the strata $x_1(\azz)=0$, $x_2(\auz)=0$, and $x_3(a_{01})=0$, which provide the versatility of $F$.} This implies that $F$ needs at least one more parameter (we will prove later that two parameters are sufficient) to achieve this transversality.
 As we have seen before, we can assume a $2$-parameter deformation $F$ of $f$ to be of the form
\begin{eqnarray}\label{eq:Fcod2a}
F(x,y,s,t)=\sum_{i+j=0}^{k}\q{A}{ij}(s,t)x^{i-j}y^j +O(k+1),
\end{eqnarray}
with $\q{A}{ij}(0,0)=\q{a}{ij}$, $\q{a}{00}=\q{a}{10}=\q{a}{11}=\q{a}{20}=\q{a}{30}=0$, $\q{a}{21}=1$ and  $\q{A}{00}(s,t)=s$.
%\begin{eqnarray}\label{eq:Fcod2a}
%	F(x,y,s,t)=\sum_{i+j=0}^{k}\q{A}{ij}(s,t)x^{i-j}y^j +O(k+1),
%\end{eqnarray}
%with $\q{A}{ij}(0,0)=\q{a}{ij}$, $\q{a}{00}=\q{a}{10}=\q{a}{11}=\q{a}{20}=\q{a}{30}=0$, $\q{a}{21}=1$ and  $\q{A}{00}(s,t)=s$, where $t=(t_1,t_2,...,t_{p-1})$.

\begin{prop}\label{prop:extratosingi}
	Let $F$ be as in \eqref{eq:Fcod2a}. Then the stratum of singularities in the parameter space is locally a regular manifold of codimention one and tangent to the manifold $\{s=0\}$.
\end{prop}
\begin{proof}
	The singular stratum we desire is given by 
	$$
	\Sigma=\{(s,t)\in \mathbb{R}^2|\exists (x,y)\in \mathbb{R}^2 ; F(x,y,s,t)=\der{F}{x}(x,y,s,t)=\der{F}{y}(x,y,s,t)=0\}.
	$$
	Since $\q{A}{00}(s,t)=s$, we have $\der{F}{s}(0,0)=1$. By the implicit function theorem, we can locally parametrize the zero set of $F$ by $
	(x,y,t)\mapsto (x,y,S(x,y,t),t).
	$
	That is, $$F(x,y,S(x,y,t),t)=0,$$ for every $(x,y,t)$ near the origin.
	Note that, by differentiating the above equation, we can find the derivatives of $S$ at the origin. For example, $\der{S}{x}(0,0,0)=\der{S}{y}(0,0,0)=0$.
	
	Now let $F_1(x,y,t)=\der{F}{x}(x,y,S,t)$, where we identify $S=S(x,y,t)$. Thus, from $\der{S}{y}(0,0,0)=0$, we obtain
	$$
	\der{F_1}{y}(0,0,0)=1.
	$$
	Again, by the implicit function theorem, the zero set of $F_1$ is locally a regular manifold that can be parametrized by
	$
	(x,t)\mapsto (x,Y(x,t),t).
	$
	That is, 
	$$F_1(x,Y(x,t),t)=0,$$ for every $(x,t)$ near the origin.
	Again, by differentiating the above equation, we can find the derivatives of $Y$ at the origin, such as $\der{Y}{x}(0,0)=0$. 
	
	Let $F_2(x,t)=\der{F}{y}(x,Y,S,t)$. Thus, from $\der{S}{x}(0,0,0)=\der{Y}{x}(0,0)=0$, we obtain
	$$
	\der{F_2}{x}(0,0,0)=1.
	$$
	Again by the implicit function theorem, the zero set of $F_2$ is locally a regular manifold that can be parametrized by
	$$
	t\mapsto (X(t),t).
	$$
	That is, $F_2(X(t),t)=0$, for every $t$ near the origin.
	Thus, we then have that $\Sigma$ is given by the manifold parameterized by
	$$
	\Xi(t)=(S(X(t),Y(X(t),t),t),t),
	$$
	which is regular, and observing that the first jet of $\Xi$ is null in the first coordenate, $\Xi$ is tangent to the manifold $\{s=0\}$.
\end{proof}
By the proposition above, we can make a change of coordinates in the parameter space such that the singular stratum is the manifold $\{s=0\}$ itself. Note that, as $A_1^-$ singularities are stable, the entire singular stratum is composed of $A_1^-$ singularities. Moreover, for each point $(0,t)$ in the singular stratum, we can continuously perform rotations, translations, and homotheties on the parameter space such that the origin is always the singularity and the curve $\alpha_1$ is always tangent to the $x$-axis. So we can assume
\begin{equation}\label{eq:simpl}
	\q{A}{00}(0,t)=\q{A}{10}(0,t)=\q{A}{11}(0,t)=\q{A}{20}(0,t)=0 \ \ \text{and}\ \  \q{A}{21}(0,t)=1.
\end{equation}

%After applying all these simplifications, we shall establish the sufficiency of two parameters for achieving geometric universality in this scenario and subsequently ascertain the condition requisite for this universality.

\begin{prop}\label{prop:transvcod2ii}
Let $F$ be as in \eqref{eq:Fcod2a} and satisfying \eqref{eq:simpl}. Let $\mathcal{\tilde{I}}$ the estratum of the curves which is a $A_1^-$ singularity and one of the branches has an first-order inflexion at singularity. Then $\Phi_f$ is transversal to the stratum $\mathcal{L}_0\cap \mathcal{\tilde{I}}$ at origin if and only if $\der{\q{A}{30}}{t}(0,0)\neq0$.
\end{prop}
\begin{proof}
The derivative matrix of $\Phi_f$ at the origin is given by
	\begin{equation}\label{eq:dPhi}
		\left(	\begin{array}{cccc}\vspace{.3cm}
			0&0& 1&0 \\\vspace{.3cm}
			0&1& \der{\q{A}{10}}{s}(0,0)&0\\\vspace{.3cm}
			1&2 \q{a}{22}& \der{\q{A}{11}}{s}(0,0)&0\\\vspace{.3cm}
			0&2 \q{a}{31}& 2\der{\q{A}{20}}{s}(0,0)&0\\\vspace{.3cm}
			2 \q{a}{31}&2 \q{a}{32}& \der{\q{A}{21}}{s}(0,0)&\der{\q{A}{21}}{t}(0,0)\\\vspace{.3cm}
			2 \q{a}{32}&6 \q{a}{33}& 2\der{\q{A}{22}}{s}(0,0)&2\der{\q{A}{22}}{t}(0,0)\\\vspace{.3cm}
			24 \q{a}{40}&6 \q{a}{41}& 6\der{\q{A}{30}}{s}(0,0)&6\der{\q{A}{30}}{t}(0,0)\\\vspace{.3cm}
			6 \q{a}{41}&4 \q{a}{42}& 2\der{\q{A}{31}}{s}(0,0)&2\der{\q{A}{31}}{t}(0,0)\\\vspace{.3cm}
			4 \q{a}{42}&6 \q{a}{43}& 2\der{\q{A}{32}}{s}(0,0)&2\der{\q{A}{32}}{t}(0,0)\\\vspace{.3cm}
			6 \q{a}{43}&24 \q{a}{44}& 6\der{\q{A}{33}}{s}(0,0)&6\der{\q{A}{33}}{t}(0,0)
		\end{array} \right).
	\end{equation}
	The stratum $\mathcal{\tilde{I}}$ near $P$ is tangent to $\{\auz=\auu=a_{30}=0\}$. So, for $\Phi_f$ to be transversal to $\mathcal{L}_0\cap\mathcal{\tilde{I}}$, it is necessary and sufficient to exist four columns of the above matrix being linearly independent with the set of vectors $\{e_4,e_5,e_6,e_8,e_9,e_{10}\}$ in $\mathbb{R}^{10}$, where $\{e_i\}_{i=1,...,10}$ is the canonical basis of $\mathbb{R}^{10}$.
	It is easy to notice that it is equivalent to $\der{\q{A}{30}}{t}(0,0)\neq0$.
\end{proof}

%\begin{lem}
%	É possível parametrizar explicitamente as variedades que representam $\mathcal{I}_2$ em $J^k(2,1)$. Estas são dadas por
%	$$(x_3,x_4,x_5,x_6,x_8,x_9,x_{10})\mapsto(0, \frac{x_3 \left( x_5 \pm \sqrt{x_5^2 - 4 x_4 x_6} \right)}{2 x_6}, x_3, x_4, x_5, x_6, X_7, x_8, x_9, x_{10}),$$
%	onde $X_7=\begin{array}{r}
%		\frac{1}{2} ( \frac{x_{10} x_5^3}{x_6^3} - \frac{3 x_{10} x_4 x_5}{x_6^2} \pm \frac{x_{10} x_5^2 \sqrt{x_5^2 - 4 x_4 x_6}}{x_6^3} \mp \frac{x_{10} x_4 \sqrt{x_5^2 - 4 x_4 x_6}}{x_6^2} +\frac{x_5 x_8}{x_6}\\ \pm \frac{\sqrt{x_5^2 - 4 x_4 x_6} x_8}{x_6} - \frac{x_5^2 x_9}{x_6^2} + \frac{2 x_4 x_9}{x_6} \mp \frac{x_5 \sqrt{x_5^2 - 4 x_4 x_6} x_9}{x_6^2})
%	\end{array}$.
%	Com isso e uma análise direta, obtemos que uma destas variedades contém o ponto $\Phi_f(0)$, enquanto a outra contém este ponto se e somente se $\eta_2=0$.
%	
%\end{lem}
%
%Given the Proposition \ref{prop:transvcod2ii}, we say that a two parameters deformation of a Morse germ function $f:(\R^2,0) \to \R$ with $j^2f$ as in (\ref{eq:fA1-simplif}) and $\atz=0$ is \textit{FRS-versal} if and only if $\der{\q{A}{30}}{t}(0,0)=0$.

Let $F$ be a two parameters deformation of $f$ as in \eqref{eq:Fcod2a} and satisfyng \eqref{eq:simpl}. By Proposição \ref{prop:I2-V2} and the hypotheses $\atz=0$ and $\eta_2\neq0$, we know that the stratum of second order inflections in $\mathbb{R}^4$ is locally composed at the origin by a single regular manifold of codimension three an the stratum of the second order vertices is empty (besides the singular stratum).

Thus, we shall find an explicit form for the projection of the second order inflection curve onto the parameter space. We know that the stratum of first order inflections is given by the zeros of the functions $F$ and
$I$. Moreover, the stratum of second-order inflections is given by $F=I=I_2=0$, where

$$\begin{array}{ll}
	I_2=&-3F_{yy}^2F_{x}^3-3F_{y}F_{yy}F_{xy}(-3F_{x}F_{xy}+F_{y}F_{xx})\\&+ F_{y}(F_{yyy}F_{x}^3-F_{y}(3F_{x}^2F_{xyy}+F_{x}(6F_{xy}^2-3F_{y}F_{xxy})+F_{y}(-3F_{xy}F_{xx}+F_{y}F_{xxx})))
\end{array}$$
in $\mathbb{R}^4$. 
 
From the proof of Proposition \ref{prop:extratosingi}, there exists a smooth function $S(x,y,t)$ such that $F(x,y,S(x,y,t),t)=0$, for every $(x,y,t)$ near the origin. This implies that, to find the first-order inflections in $\mathbb{R}^4$, it suffices to find the zeros of the functions $\bar{I}_1(x,y,t)=I_1(x,y,S,t)$, where we identify $S(x,y,t)$ simply by $S$.

By direct computations, we obtain that the second jet of $\bar{I}_1$ is given by 
$$
j^2(\bar{I}_1)=2y(x+a_{22}\,y)
$$
Since we already know that the zeros of $\bar{I}_1$ are two smooth manifolds of codimension two, observing the above second jet, we can conclude that the zeros of $\bar{I}_1$ can be parametrized by $(x,t)\mapsto(x,Y(x,t),t)$ and $(y,t)\mapsto(X(y,t),y,t)$. Clearly, using the fact that $\bar{I}_1(x,Y,t)$ and $\bar{I}_1(X,y,t)$ are identically zero, we can calculate the derivatives of $Y(x,t)$ and $X(y,t)$.

Now, find the manifold representing the stratum of second-order inflections in $\mathbb{R}^4$ is equivalent to find the zeros of the functions
$\tilde{I}_2(x,t)=I_2(x,Y(x,t),S(x,Y(x,t),t),t)$ and
$\tilde{I}_3(y,t)=I_2(X(y,t),y,S(X(y,t),y,t),t)$ near the origin, which we know that one of them is a manifolds of codimension 1 in $\mathbb{R}^2$ and the other one is empty near origin. We also know that the singular stratum shall appear in its zero sets.

Direct calculations show that 
$$
j^5(\tilde{I}_3)=-6\add x^4\left(\eta_2 + \cdots x+\cdots t\right)
$$
and
$$
j^5(\tilde{I}_2)=-6x^4\left(4a_{40}\,x+\frac{\partial A_{30}}{\partial t}(0,0) \,t\right)
$$
Recalling that the singular stratum is given by ${x=0}$ and noting that it is prominently featured in both jets shown above, we can conclude that the zeros of $\tilde{I}_3$ do not pass through the origin (except on the singular stratum), and that locally, the zeros of $\tilde{I}_2$ can be expressed as $t\mapsto(X(t),t)$, since $a_{40}$ does not vanish.
	
Hence, the stratum of second-order inflections can be parameterized by
$$
t\mapsto(X(t),Y(X(t),t),S(X(t),Y(X(t),t),t),t).
$$
Projecting onto the parameter space - last two coordinates - we see that such stratum is a smooth curve tangent to the $t$-axis. Direct calculations show that the fourth jet of this curve is given by

\begin{equation}\label{eq:inf2}
	\left(\dfrac{{\frac{\partial A_{30}}{\partial t}}(0,0)^4}{256 \,a_{40}^3} \, t^4,t\right).
\end{equation}
Assuming $\frac{\partial A_{30}}{\partial t}(0,0)\neq0$, then the stratum of second-order inflections is a smooth curve with fourth-order tangency with the stratum of singularities (See Figure \ref{fig:bifurcacao}).

\begin{figure}[h!]
	\centering
	\includegraphics[width=4cm]{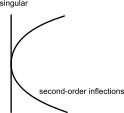}
	\caption{Geometric bifurcation set of a $A_1^-$-singularity when one of the branches has a inflection at the singular point.}%, when the helicoidal surface is not a revolution surface.}
\label{fig:bifurcacao}
\end{figure}

We have previously characterized the bifurcation set of \(F\). We shall now proceed to analyze the behavior of $F_{(s,t)}=0$ within individual strata.

By construction, for each $(t,0)$ near the origin, the zeros of $F_{(t,0)}$ are two branches, one of them being parallel to the $x$-axis at the origin. Thus, for each $t$ near the origin, there exists a function $y_t(x)$ such that $F(x,y_t(x),0,t)$ is identically zero. Let $y(x,t)=y_t(x)$. \linebreak Using the fact of that $F(x,y(x,t),0,t)$ is identically zero and direct calculations, we can conclude that $\dfrac{\partial y^3}{\partial x^2 \partial t}(0,0)=-2\frac{\partial A_{30}}{\partial t}(0,0)$ and $\dfrac{\partial y^3}{\partial x^3}(0,0)=-6 a_{40}=-6\eta_3$ (both nonzero, by hypothesis).

The first identity above reveals that, if $\frac{\partial A_{30}}{\partial t}(0,0)>0$, then the concavity of $y_t$ at the origin changes from positive to negative as $t$ changes from negative to positive. The second identity, in turn, together with \eqref{eq:inf2}, indicates the asymptotic behavior of $\alpha_1$ near the origin of the parameters space and also the behavior of the second-order inflection curve at the bifurcation set.

We still need to verify the local behavior of the other branch of the zero set of $F$ over the $t$-axis. We already know that this branch is transversal to the $x$-axis and that there exists a function $x_t(y)$ such that $F(x_t(y),y,0,t)$ is identically zero. It follows from direct calculations that $x_0'(0)=-a_{22}$ and $x_0''(0)=2\eta_2$, which does not vanish by assumption. Thus, the concavity of $\alpha_2$ depends directly on $\eta_2$.

To conclude the study of the behaviors in each stratum, we proceed as the following: First we consider the horizontal section \(t=0\), where we can verify the appearance of inflections and vertices as in \cite{GD}, case \(H5\), and which can occur in two different ways, depending only on the sign of \(\eta_3 \eta_4\).

For each of these cases, we analyze the sections \(t=k\), \(k\neq0\), and finally the behavior over the stratum of second-order inflections. Clearly, as we pass through this stratum, the second-order inflection turns into two first-order inflections on one side of the stratum, and none on the other side.

With all this analysis, we have two possible behaviors of the deformation \(F\) - up to isometries in the parameter space and/or in their domains - well defined and given in Figures \ref{fig:bifurcacaovert} and \ref{fig:bifurcacaovert2}, with their only difference being the sign of the product \(\eta_3 \eta_4\). 

\begin{figure}[h!]
\centering
\includegraphics[width=7.6	cm]{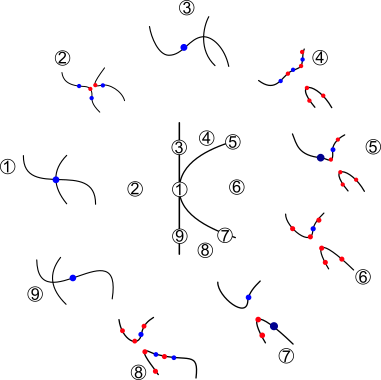}
\caption{Bifurcation set when $\eta_3\eta_4>0$. Red cdots are vertices, blue cdots are inflections and purple cdots are second-order inflections}%, when the helicoidal surface is not a revolution surface.}
\label{fig:bifurcacaovert}
\end{figure}
\begin{figure}[h!]
\centering
\includegraphics[width=8cm]{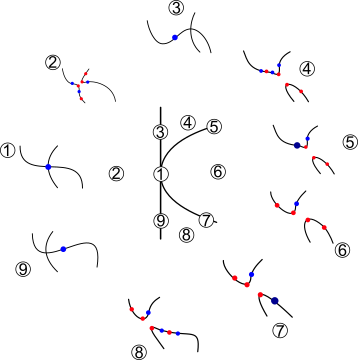}
\caption{Bifurcation set when $\eta_3\eta_4<0$. Red cdots are vertices, blue cdots are inflections and purple cdots are second-order inflections.}%, when the helicoidal surface is not a revolution surface.}
\label{fig:bifurcacaovert2}
\end{figure}
 Since, in the above approach, we only used the hypothesis of $F_{(s,t)}$ being $\mathcal{A}$-versal and $\frac{\partial A_{30}}{\partial t}(0,0)\neq0$, we conclude the following result:

\begin{theo}\label{teo:cod2i} 
	Let $f:(\mathbb{R}^2,(0,0))\to(\mathbb{R},0)$ be a smooth function germ with an $A_1^-$ singularity as in \eqref{eq:fA1-simplif}, where $\eta_1=0$ and $\eta_2,\eta_3,\eta_4$ are nonzero. Then, generically, any 2-parameter $\mathcal{A}$-versal deformation of $f$ is $FRS$-equivalent to $F:(\mathbb{R}^4,0)\to(\mathbb{R},0)$ with $F(x,y,s,t)=f(x,y)+ x^3\, t+s$. The possible configurations over the parameter space of $F_{(s,t)}$ are given in Figures \ref{fig:bifurcacaovert} and \ref{fig:bifurcacaovert2}.
\end{theo}

\subsubsection{One of the branches has a vertex at the origin}

Let us now turn to the study of the second and last geometric case of codimension two of a $A_1^-$-singularity. Without loss of generality, we assume that $\alpha_1$ is the curve with a first-order vertex at the origin and $\alpha_2$ has neither a vertex nor an inflection at the origin. This is equivalent to stating that, in \eqref{eq:fA1-naosimplif}, $\eta_3=0$, and $\eta_1,\eta_2,\eta_4$ are not zero. 

As we have seen, we can assume $F$ to be of the form
\begin{eqnarray}\label{eq:Fcod2ii}
	F(x,y,s,t)=\sum_{i+j=0}^{k}\q{A}{ij}(s,t)x^{i-j}y^j +O(k+1),
\end{eqnarray}
with $\q{A}{ij}(0,0)=\q{a}{ij}$, $\q{a}{00}=\q{a}{10}=\q{a}{11}=\q{a}{20}=0$, $\q{a}{21}=1$, $\q{a}{40}=-\add \atz^2+\atz \atu$, and $\q{A}{00}(s,t)=s$.

\begin{prop}\label{prop:extratosing}
	Let $F$ be as in \eqref{eq:Fcod2ii}. Then the stratum of singularities in the parameter space is locally a regular manifold of codimention one and tangent to the manifold $\{s=0\}$.
\end{prop}
\begin{proof} The proof of this result follows analogously to the proof of Proposition \ref{prop:extratosingi}.
\end{proof}

By the above proposition and for the same reasons as the previous case, we can assume that the singular stratum is the manifold $\{s=0\}$ and also
\begin{eqnarray}\label{eq:simpl2}
	\q{A}{00}(0,t)=\q{A}{10}(0,t)=\q{A}{11}(0,t)=\q{A}{20}(0,t)=0 \ \ \text{and}\ \  \q{A}{21}(0,t)=1.
\end{eqnarray}

\begin{prop}\label{prop:geomversal2}
	Let $F$ be as in \eqref{eq:Fcod2ii} and satisfying \eqref{eq:simpl2}. Then $\Phi_f$ is transversal to the stratum $\mathcal{S}\cap \{\eta_3=0\}$ if and only if $\der{\zeta}{t_i}(0,0)=0$, for some $i=1,...,n-1$, where $\zeta(s,t)=\q{A}{22}(s,t)\q{A}{30}^2(s,t)-\q{A}{30}(s,t) \q{A}{31}(s,t)+\q{A}{40}(s,t)$.
\end{prop}
\begin{proof} The proof of this result follows analogously to the proof of Proposition \ref{prop:transvcod2ii}.
\end{proof}

Let $F$ be a deformation of $f$ as in \eqref{eq:Fcod2ii} and satisfying \eqref{eq:simpl2}.
By Proposition \ref{prop:I2-V2}, we know that the stratum of second-order vertices in $\mathbb{R}^4$ is locally composed at the origin by a single regular manifold of codimension 3 and the stratum of second order inflections is empty (both considering beside the singular stratum).
We seek now for an explicit form for its projection onto the parameter space. We know that the stratum of first-order vertices is given by the zeros of the functions $F$ and $V_f$, where $V_f(x,y,t)=v_f(x,y)$
and the stratum of second-order vertices is given by $F=V_f=V_{f_2}=0$, where $V_2(x,y,t)=v_{f_2}(x,y)$, all of this sets in $\mathbb{R}^4$.

By the proof of Proposition \ref{prop:extratosing}, there exists a smooth function $S(x,y,t)$ such that $F(x,y,S(x,y,t),t)=0$, for every $(x,y,t)$ near the origin. This implies that, to find the stratum of first-order vertices in $\mathbb{R}^4$, it is sufficient to find the zeros of the map $\bar{V}_1(x,y,t)=V_1(x,y,S,t)$, where we identify $S(x,y,t)$ simply by $S$.

It turns out that, by direct computations, we obtain that the $2$-jet of $\bar{V}_1$ is given by 
$$
j^2(\bar{V}_1)=-6y(x+a_{22}y)(x^2+2\q{a}{22}xy-y^2).
$$
Since we already know that $\bar{V}_1=0$ are four smooth manifolds of codimension one, by observing the $2$-jet above, we can conclude that two branches of $\bar{V}_1=0$ can be parameterized by $(x,t)\mapsto(x,Y(x,t),t)$ and $(y,t)\mapsto(X(y,t),y,t)$ (The other two curves are transverse to the branches of \( C_f \)). Clearly, by using the fact that $\bar{V}_1(x,Y,t)$ and $\bar{V}_1(X,y,t)$ are identically zero, we can calculate the derivatives of $Y(x,t)$ and $X(y,t)$.

Now, to find the manifold representing the stratum of second-order vertices in $\mathbb{R}^4$, we need to find the zeros of the maps $\tilde{V}_2(x,t)=I_2(x,Y(x,t),S(x,Y(x,t),t),t)$ and $\tilde{V}_3(y,t)=I_2(X(y,t),y,S(X(y,t),y,t),t)$, which we know that one of them is a manifolds of codimension 3 and the other one is empty near $\Phi_f(0)$ (both of them besides the singular stratum).

Direct calculations show that 
$$
j^{10}(\tilde{V}_3)=-24\add x^9\left(\eta_4+\cdots x+\cdots t\right)
$$
and
$$
j^{10}(\tilde{V}_2)=24x^9\left(5(\atz^3 - \atz^2 \atd + \atz a_{41} - a_{50})\,x+\frac{\partial \eta_3}{\partial t}(0,0) \,t\right).
$$

From the equations above, the fact of $\{x=0\}$ being the singular stratum, and $\eta_4\neq0$, we ensure that the set expressing locally the desired stratum are the zeros of $\tilde{V}_2$, which can in turn be parameterized by $t\mapsto(X(t),t)$, since the coefficient of $x$ inside the parentheses is the condition for $\alpha_1$ to be a first-order vertex, and not of higher order (Section \ref{sec:A1-}), and therefore is non-zero. Thus, the stratum of second-order inflections can be parameterized by
$$
t\mapsto(X(t),Y(X(t),t),S(X(t),Y(X(t),t),t),t).
$$
Projecting onto the parameter space - last two coordinates - we find that such a stratum is a regular curve tangent to the $t$-axis. Direct calculations show that the $5$-jet of this curve is given by
$$
\left(\dfrac{\frac{\partial \zeta}{\partial t}(0,0)^5}{3125 (\atz^3 - \atz^2 \atd + \atz a_{41} - a_{50})^4} \, t^5,t\right),
$$
where $\zeta(t)=\q{A}{22}(s,t)\q{A}{30}^2(s,t)-\q{A}{30}(s,t) \q{A}{31}(s,t)+\q{A}{40}(s,t)$. Assuming $\frac{\partial \zeta}{\partial t}(0,0)$ non zero, then the stratum of second-order vertices is a regular curve with a tangency of order 5 with the stratum of singularities (See Figure \ref{fig:bif_cod2_caso2}).

\begin{figure}[h!]
	\centering
	\includegraphics[width=4cm]{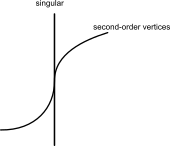}
	\caption{Geometric bifurcation set of a $A_1^-$-singularity when one of the branches has a vertice at the singular point.}
	\label{fig:bif_cod2_caso2}
\end{figure}

We already know the bifurcation set of $F$, we shall now find the behavior of the zeros of $f$ in each stratum.

Defining the functions $y(x,t)$ and $x(y,t)$ in the same way as in the previous case, we have $\dfrac{\partial y^2}{\partial x^2}(0,0)=-2\atz$, $\dfrac{\partial y^4}{\partial x^3\partial y}(0,0)=-6\frac{\partial \zeta}{\partial t}(0,0)$, and $\dfrac{\partial y^3}{\partial x^4}(0,0)=-24(\atz^2\atd-\atz\q{a}{41}+\q{a}{50}).$ (all these are nonzero, by hypotheses). 

The first equation above reveals the curvature of $y_t(x)$ near the origin. The second equation guarantees that the vertex changes sides as we pass through $t=0$ on the singular stratum. The behavior of the curve $\alpha_2$ is the same as in the previous case.

The reasoning for the analysis of all possible behaviors of the deformations on the bifurcation set is again the same as in the previous case. Note that, for the $t=0$ section, we have four possible cases, just like in \cite{GD}. These four cases are not equivalent and are shown in Figures \ref{fig:v2a}, \ref{fig:v2b}, \ref{fig:v2c}, and \ref{fig:v2d}, with their differences being the sign of the product $\eta_2 (\atz^3 - \atz^2 \atd + \atz \q{a}{41} - \q{a}{50})$ and the sign of $\eta_4$.

Since we only used the hypothesis of being \(\mathcal{A}\)-versal and $\frac{\partial \zeta}{\partial t}(0,0)\neq0$, the result follows:

\begin{theo}\label{teo:cod2ii}
	Let $f:(\mathbb{R}^2,(0,0))\to(\mathbb{R},0)$ be a smooth function germ with $A_1^-$-singularity as in \eqref{eq:fA1-simplif}, with $\eta_3=0$ and $\eta_1,\eta_2,\eta_4$ non zero. Then, generically, any 2-parameter $\mathcal{A}$-versal deformation of $f$ is $FRS$-equivalent to $F:(\mathbb{R}^4,0)\to(\mathbb{R},0)$ with $F(x,y,s,t)=f(x,y)+ x^4\, t+s$. The possible geometric configurations over the parameter space of \(F\) are given in  Figures \ref{fig:v2a}, \ref{fig:v2b}, \ref{fig:v2c}, and \ref{fig:v2d}.
\end{theo}

\begin{figure}[h!]
	\centering
	\includegraphics[width=8cm]{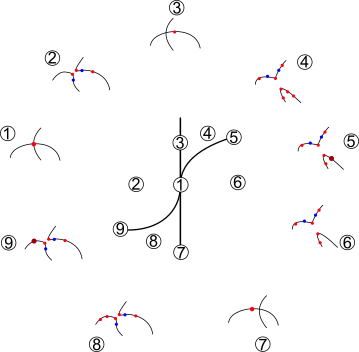}
	\caption{Bifurcation set where $\eta_2 (\atz^3 - \atz^2 \atd + \atz \q{a}{41} - \q{a}{50})>0$ and $\eta_4>0$. Red cdots are vertices, blue cdots are inflections and dark red cdots are second-order vertices.}
	\label{fig:v2a}
\end{figure}
\begin{figure}[h!]
	\centering
	\includegraphics[width=8cm]{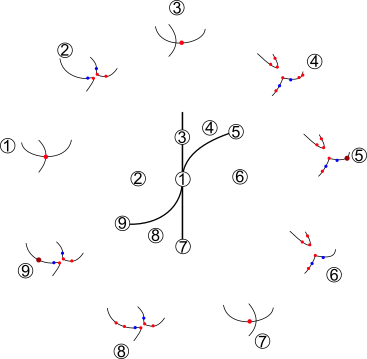}
	\caption{Bifurcation set where $\eta_2 (\atz^3 - \atz^2 \atd + \atz \q{a}{41} - \q{a}{50})<0$ and $\eta_4>0$.  Red cdots are vertices, blue cdots are inflections and dark red cdots are second-order vertices.}
	\label{fig:v2b}
\end{figure}
\begin{figure}[h!]
	\centering
	\includegraphics[width=9cm]{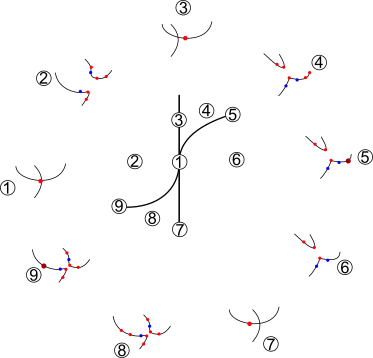}
	\caption{Bifurcation set where $\eta_2 (\atz^3 - \atz^2 \atd + \atz \q{a}{41} - \q{a}{50})>0$ and $\eta_4<0$.  Red cdots are vertices, blue cdots are inflections and dark red cdots are second-order vertices.}
	\label{fig:v2c}
\end{figure}
\begin{figure}[h!]
	\centering
	\includegraphics[width=9cm]{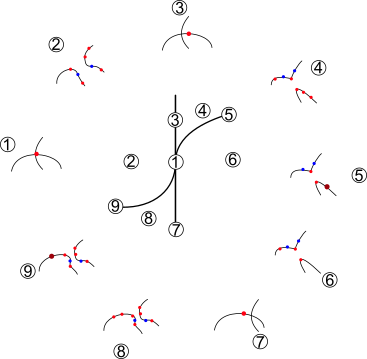}
	\caption{Bifurcation set where $\eta_2 (\atz^3 - \atz^2 \atd + \atz \q{a}{41} - \q{a}{50})<0$ and $\eta_4<0$.  Red cdots are vertices, blue cdots are inflections and dark red cdots are second-order vertices.}
	\label{fig:v2d}
\end{figure}

\section{$A_1^+$-singularity}
In this section, we investigate the geometry of a versal deformation of a Morse singularity $A_1^+$ without imposing further restrictions on it (as this would increase its codimension).
%For the same reason that for $A_1^-$ singularities of codimention one, we say that a one-parameter deformation of a generic $A_1^+$ singularity is \textit{geometrically versal}, or \textit{FRS-versal}, if it is $\mathcal{A}_e$-versal. 
Let \( f \) be an \( A_1^+ \)-singularity, that is
\begin{equation}\label{eq:fA1+naosimplif}
	f(x,y)=\sum_{i+j=0}^{k} a_{ij} x^{i-j} y^j + O(k+1),
\end{equation} 
where \( a_{ij} \) are constants and \( a_{20}^2 - 4 a_{11} a_{02} > 0 \).

From Proposition \ref{prop:I2-V2}, the stratum $\mathcal{I}_2$ is contained in the singular stratum, meaning that the curves deforming this singularity will not exhibit inflections of order two or more locally. Moreover, by the same result and analysis from $A_1^-$ singularity, the stratum of vertices will consist of two curves passing through the origin and the stratum of inflextions will be only the singularity, which means that there will be four vertices in the deformation curves and none inflections. Thus, we can conclude that

\begin{theo}\label{teo:generico}
	Let \( f : (\mathbb{R}^2, (0,0)) \to (\mathbb{R}, 0) \) be a smooth function germ with an \( A_1^+ \)singularity. Then, generically, any 1-parameter $\mathcal{A}$-versal deformation of $f$ is $FRS$-equivalent to \( F : (\mathbb{R}^3, 0) \to (\mathbb{R}, 0) \) with \( F(x,y,t) = f(x,y) + t \). The geometric configuration of $F_t=0$ is: For one sign of the parameter $t$, there are four vertices and no inflections, for the other sign, there is no curve $C_f$.
\end{theo}

\section{$A_2$-singularity}

We will now consider the geometric deformations of a $A_2$-singularity  \( f : (\mathbb{R}^2, (0,0)) \to (\mathbb{R}, 0) \). The $A_2$-singularity has codimension two. Since any additional geometric condition on $f$ implies an increase in the geometric codimension of $f$, the only case of codimension two is the generic case (when we do not put any further condictions on $f$).

It is known that a versal deformation of a $A_2$-singularity has its bifurcation set as an ordinary cusp, and the behavior of the deformed curves is given in Figure \ref{fig:A2bifurcacaosing}.

\begin{figure}[h!]
	\centering
	\includegraphics[width=9cm]{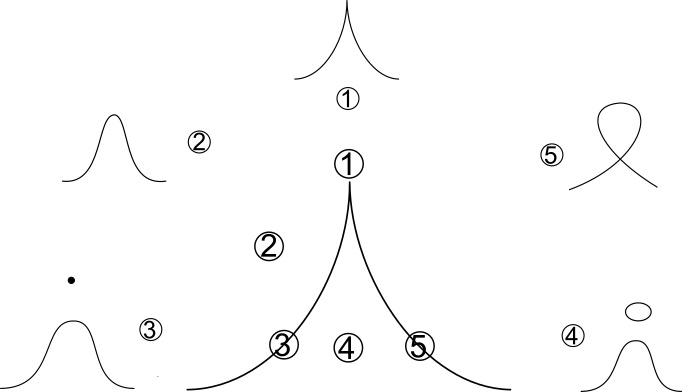}
	\caption{Bifurcation set of a versal deformation of a $A_2$-singularity.}
	\label{fig:A2bifurcacaosing}
\end{figure}

With only isometries in the domain of $f$, we can assume $f$ to be of the form

\begin{equation}\label{eq:fA2-naosimplif}
	f(x,y)=\sum_{i+j=0}^{k}\q{a}{ij}x^{i-j}y^j+O(k+1),
\end{equation} 
where $a_{ij}$ are constants, $\auu=\auz=\add=\adu=0$, $\adz=1$, and $\att\neq0$, that is, $f$ is of the form $f=x^2+\atz x^3+\atu x^2 y+\atd x y^2+\att y^3 + O(4),$ with $\att\neq0$.

It is also known that a necessary and sufficient condition for a deformation of $F$ to be versal is that the determinant of 
$$
\left[\begin{array}{cc}
	\frac{\partial A_{00}}{\partial s} & \frac{\partial A_{00}}{\partial t}\\
	\frac{\partial A_{11}}{\partial s} & \frac{\partial A_{11}}{\partial t}
\end{array}\right]
$$
is non-zero at the origin. Thus, by a change of coordinates in the parameter space, we can assume 
\begin{equation}\label{eq:versalA2}
	A00(s,t)=s \ \ \text{and} \ \ A_{11}(s,t)=t.
\end{equation}

Taking equation $i=i_2=0$, which represent the second-order inflections, we can notice that $i$ coincides with the quadratic part of $f$, applied at the point $(x,y)=(\auz,-\auu)$, while $\frac{i_2}{\auz^2+\auu^2}$ coincides with the cubic part of $f$ applied at this same point. Since $f$ is a $A_2$-singularity, the quadratic and cubic parts have no common roots, implying that locally, at $\Phi_f(0,0)$, there are no zeros of $i$ and $i_2$ simultaneously, apart from the singular stratum. %It implies that $\mathcal{V}_2$ conicides with $\mathcal{V}$ near $\mathcal{S}$.

For the zeros of $v_1=v_2=0$, we will conduct the study using resultant theory. %{\color{blue}The theory states that if the resultant between two polynomials vanishes (here we consider $\auu$ as the variable and the other coefficients as constants), then either one of the coefficients of the highest-degree terms of the polynomials vanishes, or they have common roots (which in our context implies a local intersection).VAI VOLTAR PRA SEÇÃO DOS ESTRATOS}
In this case, the resultant between $v_1$ and $v_2$ is	
$$
4096\, \auz^{33}\, \att Z,
$$
where, if we apply $\auz=\auu=0$ to $Z$, we obtain $-\atz^6$, which is the highest-degree terms of $v_1$ and $v_2$ and we can assume is non-zero. Thus, we also conclude that locally, at $\Phi_f(0)$, there is no intersection between $v_1=0$ and $v_2=0$, apart from the singular stratum itself.

Therefore, for the most general case of a $A_2$-singularity, there are no second-order inflections or vertices.

We will now prove that, generically, there is only one possible behavior for the inflections and vertices in a deformation of a $A_2$-singularity. Firstly, recall that its deformation, considering only the study of the singularities, is given in Figure \ref{fig:A2bifurcacaosing}. Over the curve of the singular stratum, its geometric behavior coincides with that for parameterized curves (see \cite{Tessier}, for more datails), so we already know how its geometry behaves (see Figure \ref{fig:A2BifParam}) .
\begin{figure}[h!]
	\centering
	\includegraphics[width=7cm]{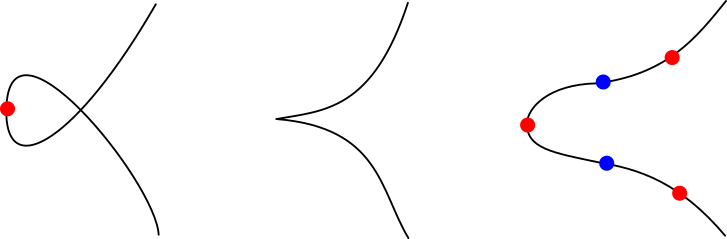}
	\caption{Versal geometric deformation of a parameterized $A_2$-singularity.  Red cdots are vertices and blue cdots are inflections.}
	\label{fig:A2BifParam}
\end{figure}

Clearly, in one of the regular parts of the singular stratum, the singularity is $A_1^-$, while in the other, $A_1^+$. The geometric transition over the $A_1^+$ part is straightforward, as there is only one possibility. Since the geometric bifurcation set of the $A_2$-singularity coincides with the bifurcation set of any versal deformation, we conclude the behavior of the deformation, which is given in Figure \ref{fig:A2bifurcacaoGeo}.

\begin{figure}[h!]
	\centering
	\includegraphics[width=8.5cm]{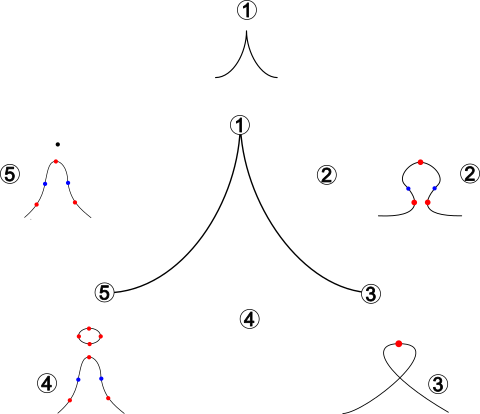}
	\caption{Versal geometric deformation of a $A_2$-singularity.  Red cdots are vertices and blue cdots are inflections.}
\label{fig:A2bifurcacaoGeo}
\end{figure}

We can thus, from the analysis conducted in this section, deduce the following result.

\begin{theo}
	Let $f:(\mathbb{R}^2,(0,0)) \to (\mathbb{R},0)$ be a smooth function germ with $A_2$ singularity. Then, generically, any 2-parameter $\mathcal{A}$-versal deformation of $f$ is FRS-equivalent to $F:(\mathbb{R}^4,0)\to(\mathbb{R},0)$ with 
	$
	F(x,y,s,t)=f(x,y)+t\,y+s.
	$
	and its geometric behavior is given in Figure \ref{fig:A2bifurcacaoGeo}.
\end{theo}

\section*{Acknowledgements}

The authors would like to thank Professor Dr. Farid Tari for his invaluable comments and significant assistance in the development of this work.

%%%%%%%%%%%%%%%%%%%%%%%%%%%%%%%%%%%%%%%%%%%%%%%%%%%%%%%%%%%%%%%%%%%%%

%%%%%%%%%%%%%%%%%%%%%%%%%%%%%%%%%%%%%%%%%%%%%%%%%%%%%%%%%%%%%%%%%%%%%%
\end{document}